\documentclass[11pt]{article}
\usepackage[pdfpagelabels,breaklinks,colorlinks,linkcolor=red,citecolor=green,urlcolor=red]{hyperref}
\usepackage[textwidth=6.5in,textheight=8.6in]{geometry}
\usepackage{latexsym}
\usepackage{amssymb}
\usepackage{amsmath}
\usepackage{algorithm}
\usepackage{amsthm}
\usepackage{caption}
\usepackage{amsfonts}
\usepackage{algpseudocode}
\usepackage{dsfont}
\usepackage[title]{appendix}
\usepackage{color,xcolor}
\usepackage{cleveref}
\usepackage{longtable}
\usepackage{multirow}
\usepackage{booktabs}
\usepackage{makecell}
\usepackage{graphicx}
\usepackage{epstopdf}
\usepackage{cases}
\usepackage{tcolorbox}
\tcbset{arc = 0pt, outer arc = 0pt,
    colback=white,
    boxsep = 0pt, left = 1pt, right = 1pt, top = 1pt, bottom = 1pt,
    boxrule = 0.5pt, bottomrule = .5pt, toprule = .5pt}

\newcommand{\RNum}[1]{\mathbf{\uppercase\expandafter{\romannumeral #1\relax}}}
\newcommand{\prox}{{\mathbf{prox}}}
\DeclareMathOperator*{\argmin}{arg\,min}
\newcommand{\conv}{\mathop{\bf conv}}
\newcommand{\RR}{\mathds{R}}

\newcommand{\cE}{{\mathcal{E}}}
\newcommand{\cO}{{\mathcal{O}}}
\newcommand{\vd}{{\rm d}}
\newcommand{\vx}{\boldsymbol{x}}
\newcommand{\vX}{\boldsymbol{X}}
\newcommand{\vy}{\boldsymbol{y}}
\newcommand{\vY}{\boldsymbol{Y}}
\newcommand{\vz}{\boldsymbol{z}}

\newcommand{\vA}{\boldsymbol{A}}
\newcommand{\vQ}{\boldsymbol{Q}}
\newcommand{\vb}{\boldsymbol{b}}
\newcommand{\vu}{\boldsymbol{u}}
\newcommand{\vr}{\boldsymbol{r}}

\colorlet{blue}{blue!90!black}
\colorlet{red}{red!50!black}
\colorlet{green}{green!50!black}

\newtheorem{assumption}{Assumption}
\newtheorem{lemma}{Lemma}
\newtheorem{theorem}{Theorem}
\newtheorem{proposition}{Proposition}
\newtheorem{corollary}{Corollary}
\newtheorem{remark}{Remark}

\title
{\bf \Large Accelerated Gradient Methods with Gradient Restart: Global Linear Convergence\footnote{The research was supported partly by the National Key R\&D Program of China (No. 2021YFA001300), the National Natural Science Foundation of China (Nos. 12271291, 12271150),
the Hunan Provincial Natural Science Foundation of China (No. 2023JJ10001),
and The Science and Technology Innovation Program of Hunan Province (No. 2022RC1190). }
}
\date{\today}
\author{
Chenglong Bao\thanks{Yau Mathematical Sciences Center, Tsinghua University, Beijing, China, and Yanqi Lake Beijing Institute of Mathematical Sciences and Applications, Beijing, China (\url{clbao@tsinghua.edu.cn}).}
\qquad
Liang Chen\thanks{School of Mathematics, Hunan University, Changsha, China, and Hunan Provincial Key Laboratory of Intelligent Information Processing and Applied Mathematics, Changsha, China (\url{chl@hnu.edu.cn}).}
\qquad
Jiahong Li\thanks{Department of Applied Mathematics, 
The Hong Kong Polytechnic University, Hong Kong, China (\url{jiahong24.li@polyu.edu.hk}).}
\qquad
Zuowei Shen\thanks{Department of Mathematics, National University of Singapore, Singapore (\url{matzuows@nus.edu.sg}).}
}

\begin{document}
\maketitle
 \begin{abstract}
Gradient restarting has been shown to improve the numerical performance of accelerated gradient methods. This paper provides a mathematical analysis to understand these advantages. First, we establish global linear convergence guarantees for both the original and gradient restarted accelerated proximal gradient
method when solving strongly convex composite optimization problems. Second, through analysis of the corresponding ordinary differential equation
model, we prove the continuous trajectory of the gradient restarted Nesterov's accelerated gradient method
exhibits global linear convergence for quadratic convex objectives, while the non-restarted version provably lacks this property by [Su, Boyd, and Cand\'es, \textit{J. Mach. Learn. Res.}, 2016, 17(153), 1-43].
\end{abstract}
\medskip
{\small
\begin{center}
\parbox{0.95\hsize}{{\bf Keywords.}\;
convex optimization,
accelerated gradient method,
restarting,
first-order methods,
global R-linear convergence rate}
\end{center}
\begin{center}
\parbox{0.95\hsize}{{\bf MSC Classification.}\; 90C25, 65K05, 65B05, 90C06, 90C30}
\end{center}}

\section{Introduction}

Consider the unconstrained convex composite optimization problem
\begin{equation}\label{obj: nonsmooth}
    \min_{\vx\in\RR^n}~F(\vx) = f(\vx) + g(\vx),
\end{equation}
where $f:\RR^n\to\RR$ is an $L$-smooth convex function and $g:\RR^n\to(-\infty,\infty]$ is a closed, proper, and convex function. The term $L$-smoothness refers to the property that $f$ is continuously differentiable and its gradient $\nabla f$ is Lipschitz continuous with $L>0$  being a Lipschitz constant.
Given an initial point $\vx_0\in\RR^n$ and set $\vy_0=\vx_0$, the celebrated accelerated proximal gradient (APG) method/the fast iterative shrinkage/thresholding algorithm ({FISTA})~\cite{beck2009fast} generates two sequences $\{\vx_k\}$ and $\{\vy_k\}$ via
\begin{equation}
\label{iter:APG}
\begin{cases}
\vx_{k+1} =\prox_{sg}(\vy_{k}-s\nabla f(\vy_{k})),  \\
\vy_{k+1} =\vx_{k+1}+\beta_{k+1}(\vx_{k+1}-\vx_{k}), 
\end{cases}
\quad k=0,1,\ldots,
\end{equation}
where $s$ is the step size, $\beta_{k+1}\in[0,1)$ is the extrapolation coefficient at the $k$-th step, and the proximal mapping $\prox_{sg}:\RR^n\mapsto\RR^n$ is defined by
\begin{equation}\label{def: prox}
\begin{array}{l}
     \prox_{sg}(\vy) := \argmin_{\vz}\big\{\frac{1}{2}\|\vy-\vz\|_2^2 + sg(\vz)|\vz\in\RR^n\big\}. 
\end{array}
\end{equation}
Generally, the extrapolation sequence $\{\beta_k\}$ in \eqref{iter:APG}  should satisfy the conditions that
\begin{equation}
\label{NestRule}
\begin{array}{l}
\beta_{k+1}=
\frac{t_{k+1}-1}{t_{k+2}}
\quad\mbox{with}\quad
t_1=1,\qquad
t_{k}\nearrow+\infty,
    \qquad
    \mbox{and}
    \qquad
    t_{k+1}^2-t_{k+1}\le t_{k}^2 
   \quad\forall k\ge 1.
\end{array}
\end{equation}
Here, $t_{k}\nearrow+\infty$ means that $t_{k+1}>t_k$ and $t_k\to\infty$ when $k\to\infty$.
In this paper, unless specifically stated otherwise, we always assume that the sequence $\{\beta_k\}$ satisfies~\eqref{NestRule}.
More specifically, there are two typical choices for this sequence.
One is from Nesteorv's original work \cite{Nesterov1983AMF} by setting
\begin{equation}
\label{step:NAG beta}
\begin{array}{l}
\beta_{k+1} = \frac{\theta _{k+1}\left( 1-\theta _k \right)}{\theta _k}
\quad\mbox{with}\quad \theta _0=1 \quad\mbox{and}\quad
\theta _{k+1}=\frac{\sqrt{\theta _{k}^{4}+4\theta _{k}^{2}}-\theta _{k}^{2}}{2}.
\end{array}
\end{equation}
The other one \cite{tseng2008accelerated, Lan2011PrimaldualFM, chambolle2015convergence} is to set
\begin{equation}
\label{step:NAG r}
\begin{array}{l}
\beta_{k+1} = \frac{k}{k+r+1}
\quad\mbox{ with the constant } r\ge 2.
\end{array}
\end{equation}
It is easy to verify that both \eqref{step:NAG beta} and \eqref{step:NAG r} are specific cases of \eqref{NestRule} by setting
$t_{k+1}=\frac{1+\sqrt{1+4t_k^2}}{2}$
and $ t_{k+1}=\frac{k+r}{r}$, respectively. 
As demonstrated in \cite{tseng2008accelerated,beck2009fast,chambolle2015convergence}, for any fixed step size $s\in(0,1/L]$, the APG method \eqref{iter:APG} with parameters adhering to \eqref{step:NAG beta} or \eqref{step:NAG r} 
has an $\mathcal{O}(1/k^2)$ convergence rate (in terms of the objective function value). 
This significantly improves the convergence of the classical gradient descent (GD) method. Furthermore, as shown in \cite{Nesterov2004IntroductoryLO}, the APG method achieves optimal complexity among all methods with only access to the first-order information of the objective function, making it a highly successful approach in optimization theory and applications.

In practical applications, the objective sequence generated by \eqref{iter:APG} with extrapolation coefficient \eqref{NestRule} may exhibit oscillation due to the increasing momentum term by $\beta_k$, potentially slowing down the convergence behavior. To further improve numerical performance, the gradient restarted APG, an adaptive restart scheme, was proposed in \cite{ODonoghue2015AdaptiveRF}. Specifically, if at the $k$-th step the following condition holds:
\begin{equation}
\langle \vx_{k+1}-\vx_k, \vy_k-\vx_{k+1}\rangle> 0,
\label{restart:gradient}
\end{equation}
one restarts the iteration~\eqref{iter:APG}, using $\vx_k$ as the initial point. It is worth noting that checking the condition~\eqref{restart:gradient} does not introduce an additional computational burden except for calculating an inner product. 
As a result, the efficient computation of \eqref{restart:gradient} and its fast numerical convergence make the gradient restarted APG method and its variants attractive choices for various applications, including image processing~\cite{bao2018coherence}, quantum tomography~\cite{shang2017superfast}, and machine learning~\cite{ito2017unified}. 
However, even $f$ is assumed to be strongly convex, the convergence rate of the gradient restarted APG method remains unknown. 
\color{black}

For the case that $g(\vx)\equiv 0$, problem \eqref{obj: nonsmooth} reduces to the smooth optimization 
$$
\min_{\vx\in\RR^n}f(\vx), 
$$
for which the iteration \eqref{iter:APG} turns to
\begin{equation}\label{NAG}
\begin{cases}
\vx_{k+1}=\vy_{k}-s\nabla f(\vy_{k}), 
\\
\vy_{k+1}=\vx_{k+1}+\beta_{k+1}(\vx_{k+1}-\vx_{k}),
\end{cases}
\quad k=0,1,\ldots,
\end{equation}
which is known as Nesterov's accelerated gradient method \cite{Nesterov1983AMF}. 
Assuming that $\{\beta_k\}$ satisfies \eqref{step:NAG beta} or \eqref{step:NAG r} with $r=2$, and setting $t=k\sqrt{s}$ and $s\to0$, one can derive the continuous counterpart of the iteration~\eqref{NAG} as the following second-order ODE~\cite{su2016differential}:
\begin{equation}
\label{low-c}
\begin{cases}
\ddot{\vX}(t)+\frac{3}{t}\dot{\vX}(t)+\nabla{f(\vX(t))}=0,  \ t>0,
\\[2mm]
\vX(0)=\vx_0,\ \dot{\vX}(0)=0.
\end{cases}
\end{equation}
The continuous model \eqref{low-c} helps to understand the APG method from different points of view, such as the fluctuation phenomena \cite[Section 3]{su2016differential} and the magic constant $3$ \cite[Section 4]{su2016differential}.
However, when the objective function is quadratic and strongly convex, \cite[Section 3.2]{su2016differential} establishes a lower bound for the residual of the objective function by
\begin{equation}
\label{eq:suupbound}
\begin{array}{ll}
    \limsup\limits_{t\to \infty} t^3\left(f\left(\vX(t)\right)-f^\star \right) \geq \frac{2\|\vx_0-\vx^\star\|^2}{\pi \sqrt{L}},
\end{array}
\end{equation}
where $f^\star$ is the optimal value and $x^\star$ is the optimal solution.
%where $\vx^\star$ is the minimizer of $f$ and  $f^\star=f(\vx^\star)$.
According to \eqref{eq:suupbound}, ${\cO}(1/k^3)$ is the best possible worst-case convergence rate of the APG method that can be obtained from the ODE model \eqref{low-c}, even when $f$ is strongly convex and $g$ is vacuous.
This prevents the analysis of the R-linear convergence rates of \eqref{low-c} for strongly convex problems. As explained in \cite[Section 5]{su2016differential}, such a drawback results from the increasingly too much momentum as $t\to\infty$, and periodic oscillations in the solution of the asymptotic ODE $\ddot{\vX}(t)+\nabla f(\vX(t))=0$ are considered the main reason impeding the convergence rate. With $g$ being vacuous, the restart condition \eqref{restart:gradient} is equivalent to $\langle \nabla f(\vy_k), \vx_{k+1}-\vx_k\rangle >0$, whose continuous counterpart is given by
\begin{equation}
\label{eq:gradrest}
    \big\langle \nabla f(\vX(t)),\dot{\vX}(t) \big \rangle \ge  0.
\end{equation}
Restarting the flow in equation \eqref{low-c} when the condition \eqref{eq:gradrest} is satisfied seems to break through the limitation of the bound \eqref{eq:suupbound}. Still, there is no theoretical validation for this claim. 

In summary, the APG method is troubled by the oscillation phenomenon, but the gradient restart scheme has shown fast numerical improvement and easy implementation in various applications. This strongly motivates us to provide a mathematical understanding of gradient restarting. Specifically, the question is whether the APG method with gradient restart schemes achieves enhanced linear convergence rates for solving strongly convex optimization problems, in discrete or continuous time perspectives.

\subsection*{Related work}
Extrapolation is a classical technique to accelerate the GD method, dating back to the heavy-ball method \cite{polyak1964some} for minimizing smooth and strongly convex functions $f$ with the updating formula 
\begin{equation}\label{iter:HB}
\begin{array}{l}
     \vx_{k+1} = \vx_k - s\nabla f(\vx_k) + \beta (\vx_k-\vx_{k-1}), 
\end{array}
\end{equation}
where $s$ and $\beta$ are positive constants. By making slight modifications to \eqref{iter:HB}, Nesterov proposed the well-known accelerating method \cite{Nesterov1983AMF} that consists of the iteration \eqref{NAG}, which is the original formula for APG with $g=0$. If $\beta_k$ follows the \eqref{step:NAG beta} or \eqref{step:NAG r}, the APG method achieves the optimal iteration complexity up to a constant among the first-order methods for convex problems~\cite{Nesterov2004IntroductoryLO}. The optimized gradient method (OGM) \cite{kim2016optimized} is a first-order algorithm similar to APG with $g=0$. It is specially designed for solving the $L$-smooth convex problem via the Performance Estimation Problem approach \cite{drori2014performance} and achieves the lower bound of complexity \cite{drori2017exact, Kim2017OntheConvergence}, which is half of APG. \cite{jang2023computer} further extended the OGM to the composite scenario ($g\in\Gamma_0$), named Optimal Iterative Shrinkage Thresholding Algorithm (OptISTA). Although the OptISTA algorithm was theoretically superior to the APG method, numerical experiments did not validate its efficiency.

For the case that $f$ is $\mu$-strongly convex, i.e., $f(\vx)-\mu\|x\|^2/2$ is convex with  $\mu>0$ being available, a preferred choice of the extrapolation parameters is setting
\begin{equation}\label{step:NAG-sc}
\begin{array}{ll}
    \beta_k = \beta^*:= \frac{\sqrt{L}-\sqrt{\mu}}{\sqrt{L}+\sqrt{\mu}}\quad\mbox{and}\quad s=1/L.
\end{array}
\end{equation}
When $g\equiv 0$, the resulting algorithm \eqref{NAG} possesses a linear convergence property that
$$
\begin{array}{ll}
f(\vx_{k+1}) - f^\star \leq \Big(1-\sqrt{\frac{\mu}{L}}\Big)^{k+1}
\left(f(\vx_0)-f^\star+\frac{\mu}{2}\|\vx_0-\vx^\star\|^2\right),
\end{array}
$$
where $\vx^\star$ is the unique minimizer of $f$ and $f^\star=f(\vx^\star)$. Similarly, using \eqref{step:NAG-sc}, Nesterov \cite{Nesterov2013GradientMF} established analogous linear convergence for \eqref{iter:APG}. 
Besides, various accelerated schemes for the strongly convex case have been proposed~\cite{siegel2021accelerated, aujol2022convergence,Aujol2022ConvergenceRO,park2021factorsqrt2,Taylor2021AnOG,ochs2014ipiano,Chambolle2016AnIT}. However, a common requirement across these methods is the estimation of the parameter $\mu$, a task that can be challenging or time-consuming in practice~\cite{Nesterov2013GradientMF, Fercoq2019AdaptiveRO, Lin2014AnAA}. Moreover, empirical observations suggest that inaccurate estimation of $\mu$ can significantly slow down the convergence speed \cite{ODonoghue2015AdaptiveRF}. Therefore, a straightforward approach is to use the off-line formulas \eqref{step:NAG beta} or \eqref{step:NAG r}, which are independent of $\mu$ and $L$. Some recent work~\cite{tao2016local,liang2017activity,li2024linear}demonstrated the local R-linear convergence of the APG method~\eqref{iter:APG} for minimizing strongly convex problems. We will further show the global R-linear convergence of the APG method, which helps to provide the convergence analysis for restarted methods.

It is widely observed that the extrapolation in the APG method introduces the oscillations, which lead to a slower convergence rate near the optimal point compared to the gradient descent (GD) or proximal gradient (PG) methods~\cite{tao2016local, liang2017activity}. This phenomenon motivates the development of methods aiming to prevent the iteration trajectory from oscillating. Recent works have proposed three main types of measures: lazy starting, monotonizing, and restarting. The lazy starting scheme is proposed in \cite{liang2022improving}, which delays oscillation occurrence by increasing $r$ in \eqref{step:NAG r}. Aujol et al. \cite{Aujol2023FISTAIA} further refined this approach by quantitatively choosing a proper $r$ based on the desired accuracy. Specifically, given a termination error $\varepsilon$, they computed a $r_{\varepsilon}$ related to $\varepsilon$ and optimal value $F^\star$ such that the APG method using formula \eqref{step:NAG r} with $r=r_{\varepsilon}$ reaches the $\varepsilon$-solution faster than PG. However, these methods can not obtain the convergence rate due to the relation between $r$ and $\varepsilon$. Additionally, determining $F^\star$ poses challenges for their practical implementation. Another measure is a monotone version introduced by Beck et al. \cite{beck2009fastGradient} and further studied by \cite{Wang2022ConvergenceRO}. The main idea involves comparing each step with the GD/PG method and selecting the better one, thereby ensuring monotonicity. While it is a simple scheme, it breaks the algebraic structure of the $\beta$ sequence, resulting in little practical improvement.

Restarting is a widespread technique employed in many accelerated methods,
such the conjugate gradient method \cite{powell1977restart}, generalized minimum residual (GMRES) method \cite{Saad1986GMRESAG}, Halpern iteration \cite{zhang2022efficient}, and Anderson acceleration method \cite{Zhang2018GloballyCT}.
For the APG algorithm \eqref{iter:APG},
this approach resets $\beta_k=0$ at the current step when the restart conditions are met. 
The convergence-based restarts including the fixed restart \cite{Nemirovski2007EFFICIENTMI, Nesterov2013GradientMF} and their variants \cite{Fercoq2019AdaptiveRO, alamo2019restart, alamo2019gradient, alamo2022restart, aujol2022fista, aujol2024parameter, roulet2020sharpness, Renegar2022ASN}, primarily leverage the sublinear rate \(F(\vx_k) - F^\star \leq \mathcal{O}(1/k^2)\) for APG and the strong convexity to establish a guarantee of linear convergence. Firstly, a straightforward strategy involves establishing a fixed schedule, which restarts iteration~\eqref{iter:APG} or~\eqref{NAG} after a preset number of iterations \cite{becker2011templates, Nesterov2013GradientMF, ODonoghue2015AdaptiveRF} or an exponential number of iterations \cite{Renegar2022ASN,roulet2020sharpness}. Although these methods can achieve linear convergence or optimal complexity, their configurations still require knowledge of the strong convexity parameter $\mu$ or the optimal function value. Additionally, a simple logarithmic grid search for these parameters is sufficient to guarantee near-optimal performance, as detailed in \cite{roulet2020sharpness}. However, this method still requires an estimate of $N$,  which is related to the condition number. Consequently, without precise information about the optimal value or the condition number of the objective function, obtaining an optimal restart schedule becomes difficult, making adaptive restart schemes based on current state information more attractive.

Typically, there are three kinds of adaptive restart schemes: the function scheme~\cite{ODonoghue2015AdaptiveRF, beck2009fastGradient}, the speed scheme~\cite{su2016differential}, and the gradient scheme \cite{ODonoghue2015AdaptiveRF}. The function scheme shares the same idea as the monotone version that compares the function value between two consecutive steps. If the objective function value increases, it restarts the iteration, and the current step is discarded. This direct method exhibits satisfactory performance, but it introduces additional computation of the function value, and the theoretical analysis always requires a sufficient decrease property. The speed restart condition~\cite{su2016differential} is given by $\|\vx_k-\vx_{k-1}\|<\|\vx_{k-1}-\vx_{k-2}\|$, 
and the corresponding continuous version, when $g(\vx)\equiv 0$, is given by
\begin{equation} \label{restart:speed}
\frac{\mathrm{d}\| \dot{\vX}(t)\|^2}{\mathrm{d}t}
\le 0,
\end{equation}
which, as proved in \cite{su2016differential},
guarantees the global linear convergence if $f$ is strongly convex.
It is noted that the ODE \eqref{low-c} implies
\begin{equation}
\label{eq:grsr}
  \frac{\vd f(\vX(t))}{\vd t}=\big\langle \nabla f(\vX(t)),\dot{\vX}(t) \big\rangle
 = -\frac{3}{t}\|\dot{\vX}(t)\|^2
 -\frac{1}{2}\frac{\mathrm{d}\|\dot{\vX}(t)\|^2}{\mathrm{d}t}.
\end{equation}
Therefore, compared to \eqref{eq:gradrest}, the speed scheme~\eqref{restart:speed} is more conservative and may somewhat scarify the acceleration property.
On the other hand, the gradient scheme~\eqref{eq:gradrest} is the weakest restart condition since it is equivalent to the monotonicity of the objective function due to the first equality in~\eqref{eq:grsr}. Numerical experiments in~\cite{su2016differential} validate the linear convergence of the gradient restart scheme, leaving an open problem
\cite[pp. 27]{su2016differential}: "Whether there exists a provable linear convergence rate for the gradient restart scheme when solving strongly convex problems?". 
If $f(\vx)=\frac{1}{2}\|A \vx-b\|_2^2$, our previous work~\cite{bao2018coherence} relaxed the gradient scheme \eqref{restart:gradient} and proved the local linear convergence under the metrical subregularity condition of the objective function on the solution set.
The recent preprint \cite{moursi2023accelerated} shows that, for one-dimensional functions within $\mathcal{F}_{L}$ and under certain conditions, the APG with gradient restarting can improve upon the $\mathcal{O}(1/k^2)$ bound. However, this applicability to higher-dimensional cases has not yet been determined. 
%In summary, the gradient restart scheme has shown fast numerical convergence and easy implementation in various literature, which strongly motivates us to provide a mathematical understanding of it.

In this paper, we establish the global R-linear convergence of the gradient restarted APG method and provide provable benefits over the non-restarted method. To achieve the linear convergence for gradient restarting, we establish the global R-linear convergence rate of the original APG method as a preliminary step. We provide a sufficient condition such that any restart scheme satisfying this condition achieves a linear convergence that outperforms that of the non-restarted scheme. 

% On the other hand, while the corresponding ODE of the APG method with $g=0$ cannot converge at a linear rate even when solving a strongly convex quadratic function, we prove the linear convergence of the gradient restarted ODE in the same scenario. 

% This result further verifies the benefit of the gradient restart strategy, and partly answers the open problem \cite[pp. 27]{su2016differential}.

In addition, 
%following the formulation in the seminal paper by Su, Boyd, and Cand\`es \cite{su2016differential}, 
we employ the ordinary differential equation (ODE) arguments of \cite{su2016differential} to establish the provable advantages of the gradient scheme \eqref{restart:gradient} for the APG method. 
Assuming $f$ is quadratic and convex, and restarting \eqref{low-c} at time $t$ when \eqref{eq:gradrest} holds,
we are able to show that the resulting trajectory $\vX^{\rm gr}(t)$ (c.f. \cref{sec:gradode}) has the global R-linear convergence property that
\begin{equation}
\label{convgence-linear}
\begin{array}{ll}
f(\vX^{\rm gr}(t)) - f^\star\leq  \frac{c_1 L\|\vx_0-\vx^\star\|^2}{2}e^{-c_2t},
\end{array}
\end{equation}
where $c_1>0 $ and $c_2\in(0,1)$ are two constants.
Compared to \eqref{eq:suupbound}, the result in \eqref{convgence-linear} validates the theoretical advantages of the gradient scheme \eqref{eq:gradrest} for the APG method. Moreover, the result~\eqref{convgence-linear} partially addresses the open question of whether there exists a provable linear convergence rate for the gradient restart scheme \eqref{eq:gradrest} when applied to strongly convex problems \cite[pp. 27]{su2016differential}. Building upon these results, this work provides a theoretical understanding of the inherent properties of the gradient restarting within the accelerated gradient methods, which are consistently observed in numerical experiments.

\subsection*{Organization}
The remaining part of this article is organized as follows.
In  \Cref{sec:apgc}, we focus on the APG method with gradient restarting and establish its global R-linear convergence rate for problems with strongly convex objective functions. 
In \Cref{sec:gradode}, we prove the global R-linear convergence of the solution trajectory corresponding to the ODE model~\eqref{low-c} with gradient restart scheme \eqref{eq:gradrest}, under the assumption that $f$ is quadratic and convex. 
Finally, we conclude this paper in \Cref{sec:con} with some discussions.

\subsection*{Notation}  
We use bold lowercase letters (e.g., $\vx$) to represent vectors and bold uppercase letters (e.g., $\vA$) to represent matrices.
Given a real number $u\in\RR$, we denote $\lfloor u\rfloor$ as the floor number (greatest integer less than or equal to) of $u$, and define the fractional part of $u$ as ${\bf frac}(u) = u - \lfloor u\rfloor$. For any positive integers $m$ and $i$, we define $[m]$ as the set $\{1, 2, \ldots, m\}$ and $[m] \pm i$ to be the set $\{j \pm i \mid j \in [m]\}$.

Let $\|\cdot\|$ denote the Euclidean norm on $\RR^n$ and $L>0$, we define $\mathcal{F}_{L} $ to be the class of $L$-smooth convex functions on $\RR^n$, that is
\begin{equation*}
\mathcal{F}_{L} : = \big\{f:\RR^n\to \RR\mid f\text{ is convex and differentiable, and } \|\nabla f(\vx)-\nabla f(\vy)\|\leq L\|\vx-\vy\|,\forall \vx,\vy\in\RR^n\big\}.
\end{equation*}
In addition, $\mathcal{S}_{\mu}$ denotes the class of $\mu$-strongly convex functions on $\RR^n$, that is
\begin{equation*}
\mathcal{S}_\mu : = \big\{f:\RR^n\to(-\infty,+\infty]\mid  f(\vx)-\frac{\mu}{2}\|\vx\|^2\text{ is convex}\big\}.
\end{equation*}
We define $\mathcal{S}_{\mu, L}: = \mathcal{F}_{L}\cap\mathcal{S}_\mu$ and denote $\Gamma_0$ as the class of closed, proper, convex, and extended real-valued functions on $\RR^n$.

\section{Global R-linear convergence of APG with gradient restarting}\label{sec: discrete}
This section is devoted to proving the global R-linear convergence for the APG method with gradient restarting in terms of both the objective value and the iteration sequence. We first establish the global R-linear convergence of the original APG method in \Cref{subsection:apgrate}, based on which our main results are developed in \Cref{sec:apgc}. 
For convenience, we make the following blanket assumption throughout this section. 
\begin{assumption}
\label{ass:blanket}
Assume for probem \eqref{obj: nonsmooth} that $f\in\mathcal S_{\mu, L}$, $g\in\Gamma_0$ with $L>\mu>0$. Moreover, let $\vx^\star$ denote the unique minimizer of $F$ defined in \eqref{obj: nonsmooth} and $F^{\star}=F(\vx^\star)$.
\end{assumption}

\subsection{Global R-linear convergence of APG}
\label{subsection:apgrate}

In this subsection, we establish the global R-linear convergence of the APG  method \eqref{iter:APG} under \Cref{ass:blanket}, considering both the objective value and the iteration sequence. 

Firstly, given $s>0$ and $g\in\Gamma_0$, we define the auxiliary function $G_s:\RR^n\to\RR^n$ as
\begin{equation}
\label{gsy}
G_s(\vy):=\frac{\vy-\prox_{s g}(\vy-s\nabla f(\vy))}{s}. 
\end{equation}
Then, the following basic descent property holds for the function $F$ satisfying \Cref{ass:blanket}. 
\begin{lemma}
%\label{lemma: bas2}
Under Assumption \ref{ass:blanket}, it holds for any $\vx,\vy\in\RR^n$ that
	\begin{equation}\label{bas2}
    \begin{array}{c}
         F\left(\vy-s G_{s}(\vy)\right) \leq F(\vx)+G_{s}(\vy)^{T}(\vy-\vx)-\left(s-\frac{Ls^2}{2}\right)\left\|G_{s}(\vy)\right\|^{2}-\frac{\mu}{2}\|\vy-\vx\|^{2}.
    \end{array}	
	\end{equation}
\end{lemma}

\begin{proof}
The first-order optimality of \eqref{def: prox} gives $G_s(\vy)-\nabla f(\vy)\in \partial g(\vy-sG_s(\vy))$, which implies
$$
\begin{array}{l}
		g(\vy-sG_s(\vy))-g(\vx)\leq \left( G_s(\vy)-\nabla f(\vy) \right)^{T}\left( \vy-sG_s(\vy)-\vx \right).
\end{array}
$$
Moreover, since $f\in\mathcal{S}_{\mu,L}$, it holds
\begin{equation}
    \begin{cases} 
f(\vy-sG_s(\vy))-f(\vy)\leq -s\nabla f(\vy)^{T}G_s(\vy)+\frac{Ls^2}{2}\left\| G_s(\vy) \right\|^2,\label{ineq: f dec}
\\
f(\vy)-f(\vx)\leq \nabla f(\vy)^{T}(\vy-\vx)-\frac{\mu}{2}\left\| \vy-\vx \right\|^2.
\end{cases}
\end{equation}
Combining the above three inequalities, one obtains
$$
\begin{array}{rl}
F(\vy-sG_s(\vy))-F(\vx)
=
&f(\vy-sG_s(\vy))-f(\vy)+f(\vy)-f(\vx)+g(\vy-sG_s(\vy))-g(\vx)
		\\[2mm]
		\leq &-s\nabla f(\vy)^{T}G_s(\vy)+\frac{Ls^2}{2}\left\| G_s(\vy) \right\|^2+\nabla f(\vy)^{T}(\vy-\vx)-\frac{\mu}{2}\left\| \vy-\vx \right\|^2\\[1mm]
		&+\left( G_s(\vy)-\nabla f(\vy) \right)^{T}\left( \vy-sG_s(\vy)-\vx \right),
	\end{array}
	$$
which implies \eqref{bas2}, and this completes the proof.    
\end{proof}

Based on \eqref{gsy}, one can equivalently rewrite the APG method \eqref{iter:APG} as the \Cref{algoapgc}.

\begin{algorithm}[ht]
\caption{Accelerated proximal gradient algorithm for  problem \eqref{obj: nonsmooth}}
\label{algoapgc}
\begin{algorithmic}[1]
\Require Initial point $\vx_0=\vy_0$, step size $s>0$, and parameter sequence $\{t_k\}$ satisfying \eqref{NestRule}.
\Ensure the minimizing sequence $\{\vx_k\}$ and $\{\vy_k\}$.
\For{$k=0,1,2,\ldots$ }
\State{
\begin{equation}
\label{APG}
\left\{\begin{array}{l}
\vx_{k+1}:=\vy_k-sG_s(\vy_k)\\[1mm]
\beta_{k+1}:=(t_{k+1}-1)/t_{k+2}\\[1mm]
\vy_{k+1}:=\vx_{k+1}+\beta_{k+1}(\vx_{k+1}-\vx_{k}).
\end{array}\right.
\end{equation}
}
\EndFor
\end{algorithmic}
\end{algorithm}
Next, to analyze the convergence rate of Algorithm \ref{algoapgc}, we define the Lyapunov sequence $\{  \cE_k\} $ by 
\begin{equation}
\label{lyap: APG}
\begin{array}{c}
     \cE_k:=\underbrace{s(t_{k+1}-1)t_{k+1}\left(F\left(\vx_{k}\right)-F^{\star}\right)}_{\cE_k^p}+\underbrace{\frac{1}{2}\left\|(t_{k+1}-1) (\vy_k-\vx_k)+\left(\vy_{k}-\vx^{\star}\right)\right\|^{2}}_{\cE_k^m}. 
\end{array}
\end{equation}
We have the following result on the Lyapunov sequence $\{\cE_k\} $.

\begin{lemma}
\label{lemma: diff apg}
Under Assumption \ref{ass:blanket}, let $\{\vx_k\}$ and $\{\vy_k\}$ be the sequences generated by Algorithm \ref{algoapgc} with step size $s\in(0,1/L)$ and $\{\cE_k\}$ be the Lyapunov sequence defined by \eqref{lyap: APG}. 
Then, it holds that
\begin{numcases}{~}
  \begin{array}{ll}
    \cE_{k+1}-\cE_k\leq
-\frac{s^2 t_{k+1}^2(1-sL)}{2}\left\|G_s(\vy_k)\right\|^2
 -\frac{\mu s (t_{k+1}-1)t_{k+1}}{2}\left\|\vy_k-\vx_k\right\|^2-\frac{\mu s t_{k+1}}{2}\left\|\vy_k-{\vx^\star}\right\|^2,
  \end{array}
 \label{ineq: diff of lyap apg}
\\[1mm]
    \begin{array}{ll}
        \cE_{k+1}\leq & \big[ (t_{k+2}-1)t_{k+2}\big(\frac{1}{2\mu s}-1+\frac{Ls}{2}\big)
        +\frac{(1+u+1/v)t_{k+1}^2}{2} \big] \big\| sG_s(\vy_{k}) \big\|^2
        \\[.5mm]
        &\quad +\frac{(1+v+w)(t_{k+1}-1)^2}{2}\big\|\vy_k-\vx_k\big\|^2 +\frac{1+1/w+1/u}{2}\big\|\vy_k-{\vx^\star}\big\|^2.
    \end{array}
    \label{upper bound: lyapk+1 apg}
\end{numcases}
for all $k\ge0$ and $u>0,v>0,w>0$.
\end{lemma}

\begin{proof}
Given $k\ge1$, a direct computation based on \eqref{lyap: APG} gives
\begin{equation}\label{diff part 1}
\begin{array}{l}
       \cE_{k+1}^p-\cE_k^p =s(t_{k+2}-1)t_{k+2}\left[F\left(\vx_{k+1}\right)-F^{\star}\right]-s(t_{k+1}-1)t_{k+1}\left[F\left(\vx_{k}\right)-F^{\star}\right]
       \\[1mm]
    =\underbrace{s(t_{k+1}^2-t_{k+1})\left[F\left(\vx_{k+1}\right)-F\left(\vx_{k}\right)\right]}_{\RNum{1}_1}
    +\underbrace{s(t_{k+2}^2-t_{k+2}-t_{k+1}^2+t_{k+1})\left[F\left(\vx_{k+1}\right)-F^{\star}\right]}_{\RNum{1}_2}.
\end{array}
\end{equation}
According to the facts that $t_{k+2} (\vy_{k+1}-\vx_{k+1})=(t_{k+1}-1) (\vx_{k+1}-\vx_k)$ and $\vy_{k}=\vx_{k+1}+sG_s(\vy_{k})$, one can get
\begin{equation}\label{proof:lemma2-eq1}
\begin{array}{ll}
     &  (t_{k+2}-1) (\vy_{k+1}-\vx_{k+1})+\left(\vy_{k+1}-\vx^{\star}\right)-(t_{k+1}-1) (\vy_k-\vx_k)-\left(\vy_{k}-\vx^{\star}\right)\\[.5mm]
    % = &(t_{k+2}-1) (\vy_{k+1}-\vx_{k+1})+\left(\vy_{k+1}-\vy_{k}\right)-(t_{k+1}-1) (\vy_k-\vx_k)\\
    = & (t_{k+2}-1) (\vy_{k+1}-\vx_{k+1})
    +\left(\vy_{k+1}-\vx_{k+1}-sG_s(\vy_{k})\right)
   \\[.5mm]
    &~ -(t_{k+1}-1) (\vx_{k+1}+sG_s(\vy_{k})-\vx_k)\\[.5mm]
    = &t_{k+2} (\vy_{k+1}-\vx_{k+1})-(t_{k+1}-1) (\vx_{k+1}-\vx_k)-t_{k+1}sG_s(\vy_{k})
    = -t_{k+1}sG_s(\vy_{k}). 
\end{array}
\end{equation}
Therefore, by using \eqref{proof:lemma2-eq1} and the fact $\frac{1}{2}\|\boldsymbol{\alpha}\|^2-\frac{1}{2}\|\boldsymbol{\beta}\|^2=\frac{1}{2}\|\boldsymbol{\alpha}-\boldsymbol{\beta}\|^2+\boldsymbol{\beta}^{T}(\boldsymbol{\alpha}-\boldsymbol{\beta})$ for any $\boldsymbol{\alpha} ,\boldsymbol{\beta} \in\RR^n$, one gets
\begin{equation}\label{diff part 2}
    \begin{array}{rl}
   \cE_{k+1}^m-\cE_k^m   
       =&\underbrace{- t_{k+1}(t_{k+1}-1)sG_s(\vy_{k})^{T} (\vy_{k}-\vx_{k})}_{\RNum{2}_1}\underbrace{- t_{k+1}sG_s(\vy_{k})^{T} (\vy_{k}-{\vx^\star})}_{\RNum{2}_2} \underbrace{+\frac{t_{k+1}^2}{2}\left\| sG_s(\vy_{k}) \right\|^2}_{\RNum{2}_3}.
    \end{array}
\end{equation}
By setting $\vy=\vy_k, \vx=\vx_k$ or $\vy=\vy_k, \vx=\vx^{\star}$ in \eqref{bas2}  one can readily obtain that
\begin{numcases}{~}
  \begin{array}{ll}
     F\left(\vx_{k+1}\right)-F(\vx_{k}) \leq G_s(\vy_k)^{T}(\vy_k-\vx_k)-\big(s-\frac{Ls^2}{2}\big)\big\|G_s(\vy_k)\big\|^{2}-\frac{\mu}{2}\|\vy_k-\vx_k\|^{2},   \end{array}\label{ykxk} \\[1mm]
       \begin{array}{ll}
 F\left(\vx_{k+1}\right)-F^\star  \leq G_s(\vy_k)^{T}(\vy_k-{\vx^\star})-\big(s-\frac{Ls^2}{2}\big)\big\|G_s(\vy_k)\big\|^{2}-\frac{\mu}{2}\|\vy_k-{\vx^\star}\|^{2}.  \end{array}\label{ykxstar}
\end{numcases}
Using the inequality \eqref{ykxk}, one can find an upper bound for the sum of  $\RNum{1}_1$ in \eqref{diff part 1} and $\RNum{2}_1$ in \eqref{diff part 2} that
\begin{equation}\label{ineq: 11-21}
    \begin{array}{ll}
        \RNum{1}_1+\RNum{2}_1&=s(t_{k+1}-1)t_{k+1}\left[ F(\vx_{k+1})-F(\vx_{k})-G_s(\vy_{k})^{T} (\vy_{k}-\vx_{k}) \right]\\[1mm]
        &\leq (t_{k+1}-1)t_{k+1}\Big[ -\left(1-\frac{Ls}{2}\right)\left\|sG_s(\vy_k)\right\|^{2}-\frac{\mu s}{2}\|\vy_k-\vx_k\|^{2}  \Big].
    \end{array}
\end{equation}
By \eqref{ykxstar} one can get an upper bound of the sum of
$\RNum{1}_2$ in \eqref{diff part 1} and $\RNum{2}_2$ in \eqref{diff part 2} as
\begin{equation}\label{ineq: 12-22}
    \begin{array}{ll}
         \RNum{1}_2+\RNum{2}_2
    &=s(t_{k+2}^2-t_{k+2}-t_{k+1}^2+t_{k+1})\left[F\left(\vx_{k+1}\right)-F^{\star}\right]- t_{k+1}sG_s(\vy_{k})^{T} (\vy_{k}-{\vx^\star})\\[1mm]
        &\leq st_{k+1}\left[ F(\vx_{k+1})-F^\star-G_s(\vy_{k})^{T} (\vy_{k}-{\vx^\star}) \right]
        % -s\left[t_{k+2}^2-t_{k+2}-t_{k+1}^2\right]\left[f\left(\vx_{k+1}\right)-f^{\star}\right]
        \\[1mm]
        &\leq t_{k+1}\left[ -\left(1-\frac{Ls}{2}\right)\left\|sG_s(\vy_k)\right\|^{2}-\frac{\mu s}{2}\|\vy_k-{\vx^\star}\|^{2}  \right],
    \end{array}
\end{equation}
where the first inequality follows from the Nesterov rule in \eqref{NestRule}.
Combing
\eqref{diff part 1}, \eqref{diff part 2}, \eqref{ineq: 11-21} and \eqref{ineq: 12-22} together implies
$$
\begin{array}{ll}
     \cE_{k+1}-\cE_k
     & \le ~ (t_{k+1}-1)t_{k+1}\left[ -\left(1-\frac{Ls}{2}\right)\left\|sG_s(\vy_k)\right\|^{2}-\frac{\mu s}{2}\|\vy_k-\vx_k\|^{2}  \right]\\[1mm]
    &\quad  +t_{k+1}\left[ -\left(1-\frac{Ls}{2}\right)\left\|sG_s(\vy_k)\right\|^{2}-\frac{\mu s}{2}\|\vy_k-{\vx^\star}\|^{2}  \right]+\frac{t_{k+1}^2}{2}\left\| sG_s(\vy_{k}) \right\|^2
    \\[1mm]
    &= -\frac{ t_{k+1}^2(1-sL)}{2}\left\|sG_s(\vy_k)\right\|^2-\frac{\mu s (t_{k+1}-1)t_{k+1}}{2}\left\|\vy_k-\vx_k\right\|^2-\frac{\mu s t_{k+1}}{2}\left\|\vy_k-{\vx^\star}\right\|^2,
\end{array}
$$
so that \eqref{ineq: diff of lyap apg} holds. On the other hand, taking \eqref{ykxstar} into the first part in \eqref{lyap: APG}, one has
\begin{equation}
\label{lyap nag k+1: 1st}
\begin{array}{ll}
\cE_{k+1}^p &
\le
s(t_{k+2}^2-t_{k+2})\left[
G_s(\vy_k)^{T}(\vy_k-{\vx^\star})-\left[s-\frac{Ls^2}{2}\right]\left\|G_s(\vy_k)\right\|^{2}-\frac{\mu}{2}\|\vy_k-{\vx^\star}\|^{2}
\right]
\\[2mm]
&\le
s(t_{k+2}-1)t_{k+2}\big(
\big[\frac{1}{2\mu}-s+\frac{Ls^2}{2}\big]\left\|G_s(\vy_{k}) \right\|^2
\big). 
\end{array}
\end{equation}
Meanwhile, one has for any $u>0,v>0$, and $w>0$ that
\begin{equation}\label{lyap nag k+1: 2nd}
\begin{array}{ll}
\cE_{k+1}^m
&=\frac{1}{2}\left\|(t_{k+2}-1) (\vy_{k+1}-\vx_{k+1})+\left(\vy_{k+1}-\vx^{\star}\right)\right\|^{2}
\\[1mm]
&=\frac{1}{2}\left\|t_{k+2} (\vy_{k+1}-\vx_{k+1})+\left(\vx_{k+1}-\vx^{\star}\right)\right\|^{2}
\\[1mm]
&=\frac{1}{2}\left\|(t_{k+1}-1) (\vx_{k+1}-\vx_{k})+\left(\vx_{k+1}-\vx^{\star}\right)\right\|^{2}
\\[1mm]
&
=\frac{1}{2}\left\|t_{k+1}(\vx_{k+1}-\vy_{k})+(t_{k+1}-1)(\vy_k-\vx_k)+\left(\vy_{k}-\vx^{\star}\right)\right\|^{2}
\\[1mm]
&\leq  \frac{1+u+1/v}{2}t_{k+1}^2\left\| sG_s(\vy_{k}) \right\|^2+\frac{1+v+w}{2}(t_{k+1}-1)^2\left\|\vy_k-\vx_k\right\|^2
+\frac{1+1/w+1/u}{2}\left\|\vy_k-{\vx^\star}\right\|^2.
\end{array}
\end{equation}
Summing the inequalities \eqref{lyap nag k+1: 1st} and \eqref{lyap nag k+1: 2nd} together implies \eqref{upper bound: lyapk+1 apg} and this completes the proof.
\end{proof}

Based on \Cref{lemma: diff apg}, we demonstrate the global R-linear convergence of Algorithm \ref{algoapgc} as follows.
\begin{theorem}
\label{thm: APG}
Under Assumption \ref{ass:blanket}, let $\{\vx_k\}$ and $\{\vy_k\}$ be the sequences generated by Algorithm \ref{algoapgc} with step size $s\in(0,1/L)$
and $\{\cE_k\}$ be the Lyapunov sequence defined by \eqref{lyap: APG}. 
Then, it holds for all $k\ge 0$ that
\begin{equation}
\label{lyap decrease apg}
    \cE_{k+1}\leq \rho\cE_k, \quad\mbox{with} \quad  \rho: = 1- \frac{(1-L s)\mu s}{3} \in (1-\mu s ,1),
\end{equation}
and for all $k\ge 1$ that
\begin{equation}
\label{apg linear}
F(\vx_{k}) - F^\star\le
\frac{ \rho^k \|\vx_0-\vx^\star\|^2 }{2s(t_{k+1}-1)t_{k+1}}  .
\end{equation}
Moreover, if $t_k$ is chosen as \eqref{step:NAG beta} or \eqref{step:NAG r}, it holds that
\begin{equation}
\label{vx45}
F(\vx_{k}) - F^\star \leq \mathcal{O}({\rho}^{k}/k^2).
\end{equation}
%and the term $\mathcal{O}({\rho}^{k}/k^2)$ is a tight lower bound of the convergence rate.
%right-hand term in \eqref{apg linear}.
%with $\rho\le 1- \frac{(1-L s)\mu s}{3}$.
\end{theorem}

\begin{proof}
Given $k\ge 0$, let $ D_k:\RR^3\to \RR$ be the function
\begin{equation}\label{def: dk}
\begin{array}{l}
\begin{array}{r}
 D_k(u,v,w):= \max \left\{
\frac{(t_{k+2}-1)t_{k+2}}{t_{k+1}^2(1-sL)}\left(\frac{1}{s\mu}-2+Ls\right) +\frac{(1+u+1/v)}{1-sL},
\frac{(1+v+w)(t_{i+1}-1)}{\mu s t_{k+1}},
\frac{1+1/w+1/u}{\mu s t_{k+1}}\right\},
\end{array}
\end{array}
\end{equation}
and define ${\cal D}_k:=\inf\limits_{u,v,w>0} D_k(u,v,w)$. 
From  \Cref{lemma: diff apg} we know that
$\cE_{k+1} \leq {\cal D}_k ( \cE_k-\cE_{k+1})$.
 Since $s<1/L$, one has
$\frac{1}{s\mu}+Ls\ge \frac{1}{Ls}+Ls > 2$.
Consequently, all the three terms in the brace of \eqref{def: dk} are closed proper convex functions, then so is $D_k$.
The first term in the brace of \eqref{def: dk} can imply that ${\cal D}_k\ge \frac{1}{1-Ls}$.
%Meanwhile, the third term in the brace implies that $\cD_0> \frac{1}{\mu s}$.
Moreover, according to $t_{k+1}>1$ and $t_{k+2}^2-t_{k+2}\le t_{k+1}^2$ by \eqref{NestRule} one has
\[
\begin{array}{ll}
{\cal D}_k+1&\le
D_k(1,1,1)+1 \le
1+ \max \left\{
\frac{1}{1-sL}\left(\frac{1}{s\mu}+1+Ls\right) ,\frac{3}{\mu s }
\right\}
\\[2mm]
&
\le \max \left\{
\frac{1}{1-sL}\left(\frac{1}{s\mu}+2\right) ,\frac{3+\mu s}{\mu s }
\right\}
=
\frac{1}{s\mu}\max \left\{
\frac{1+2 s\mu }{1-sL} ,{3+\mu s}
\right\}
<\frac{3}{(1-sL)\mu s}.
\end{array}
\]
Then we have
$\cE_{k+1}\leq
\big(1-\frac{1}{{\cal D}_k+1}\big)\cE_k$
and $1<{\cal D}_k+1 \le\frac{3}{\mu s(1-Ls)}$. 
By setting $\rho=1- \frac{(1-L s)\mu s}{3}\ge\sup_{i\ge1}
\{1-\frac{1}{{\cal D}_i+1}\}$, 
we have $\rho\in(1-\mu s,1)$ and \eqref{lyap decrease apg} holds. 
Taking \eqref{lyap: APG} into account yields \eqref{apg linear}.
Moreover, if $t_k$ satisfies \eqref{step:NAG beta} or \eqref{step:NAG r}, then $t_k=\mathcal{O}(k)$, so \eqref{vx45} holds. This completes the proof. 
\end{proof}

\begin{remark}\label{remark tk<k}
According to \eqref{NestRule}, the sequence $\{t_k\}$ satisfies 
$$
\begin{array}{l}
t_1=1,\quad t_{k+1}-t_k\leq {t_k}/(t_{k+1}+t_k)<1\Longrightarrow t_k<k~~\mbox{ for all } k\ge1.
\end{array}
$$
Therefore, the right-hand term in \eqref{vx45} is the tight lower bound for the right-hand term in \eqref{apg xkx0}.

\end{remark}

\begin{remark}
    The fixed step size \(s < {1}/{L}\) can be relaxed using a sequence \(\{s_k\}\) defined by a backtracking line search satisfying condition \eqref{ineq: f dec}. Without providing a proof, we assert that a similar Lyapunov analysis applies when the step size \(s_k\) is monotonically decreasing and bounded from below.
\end{remark}

\begin{remark}
Let $\{\vx_k\}$ be the sequence generated by the fixed restarted APG method, with a restart interval of $K\in\mathbf{Z}_{++}$ such that $\bar\rho:=\sqrt[K]{1/[\mu s (t_{K+1}-1)t_{K+1}]}\in(0,1)$. Then, according to \Cref{thm: APG}, it has
    $$
    \begin{array}{l}
         F(\vx_{mK}) - F^\star \le \frac{\rho^{K} }{\mu s(t_{K+1}-1)t_{K+1}}\cdot\left( F(\vx_{(m-1)K}) - F^\star \right)\le\cdots\le(\bar\rho\cdot\rho)^{mK}\left( F(\vx_{0}) - F^\star \right),
    \end{array}
    $$
    where $m\in\mathbf{Z_{++}}$ is the number of restarting.
Such a convergence result improves the linear rate from $\bar \rho$ in \cite{Nemirovski2007EFFICIENTMI, Nesterov2013GradientMF} to $\bar\rho\cdot\rho$, and the enhancement is also effective for the variants of the fixed restart strategy summarized in the Introduction. 
\end{remark}

%The convergence rate analysis of adaptive restarted schemes is much involved. 
The analysis for adaptive restarting in the next section requires the following result regarding the linear convergence of the iteration sequence generated by \Cref{algoapgc}.

%However, it is challenging to verify whether the inequality $\mu s(t_{K_i+1}-1)t_{K_i+1}>1$ consistently holds for adaptive restarted schemes, where $K_i$ represents the length of the $i$th restart interval.
%To address this issue, 
%for establishing the linear convergence of the gradient restarted APG method in the next section.
%, which addresses the previously mentioned difficulty.

\begin{proposition}
\label{thm: apgseq}
Under Assumption \ref{ass:blanket}, let $\{\vx_k\}$ and $\{\vy_k\}$ be the sequences generated by Algorithm \ref{algoapgc} with step size $s\in(0,1/L)$. 
Then, it holds for all $k\geq 1$ that
\begin{equation}
\begin{array}{c}
\|\vx_k-\vx^\star\|^2
 \leq (1-\mu s)
 \max\big\{ \frac{1}{t_k}
 , 1-\mu s (t_k-1) \big\}
 \cdot
  \rho^{k-1}
 \left\|\vx_{0}-\vx^\star\right\|^{2} 
    \end{array}   
\label{apg xkx0}
\end{equation}
with $\rho=1- {(1-L s)\mu s}/{3} \in \left(1-\mu s,1\right)$.
Moreover, if $t_k$ is chosen as \eqref{step:NAG beta} or \eqref{step:NAG r}, it has
\begin{equation*}
\begin{array}{l}
     \|\vx_k-\vx^\star\|^2
\leq \mathcal{O}({\rho}^{k-1}/k), 
\end{array}
\end{equation*}
and the term $\mathcal{O}(\rho^{k-1}/k)$ is a tight lower bound for the right-hand term in \eqref{apg xkx0}.
\end{proposition}

\begin{proof} 
We prove \eqref{apg xkx0} by induction. Given $k\ge0$, according to \eqref{ykxstar} and $s < 1 / L$, one has
$$
\begin{array}{ll}
&F\left(\vx_{k+1}\right)-F(\vx^\star)
\leq G_s(\vy_k)^T(\vy_k-\vx^\star)-\frac{s}{2}\left\|G_s(\vy_k)\right\|^2-\frac{\mu}{2}\left\|\vy_k-\vx^\star\right\|^2
\\
=&\frac{1}{2s}\left(\|\vy_k-\vx^\star\|^2-\|\vx_{k+1}-\vx^\star\|^2\right)
-\frac{\mu}{2}\|\vy_k-\vx^\star\|^2
=\frac{1}{2s}\left((1-\mu s)\|\vy_k-\vx^\star\|^2-\|\vx_{k+1}-\vx^\star\|^2 \right). 
\end{array}
$$
Since $F\left(\vx_{k+1}\right)\ge F(\vx^\star) $, we have
\begin{equation}\label{PG descent}
\begin{array}{l}
     \|\vx_{k+1}-\vx^\star\|^2\leq (1-\mu s)\|\vy_k-\vx^\star\|^2. 
\end{array}
\end{equation}
Note that $t_1=1$ and $\vy_0=\vx_0$, the equality \eqref{PG descent} implies that \eqref{apg xkx0} holds when $k=1$. Assume that \eqref{apg xkx0} holds for a certain $k\ge2$. Then, according to $t_k\ge1$, $\vx_k$, one has $ \left\|\vx_k-\vx^\star\right\|^2 \leq\left(1-\mu s\right)
\rho^{k-1}
\left\|\vx_{0}-\vx^\star\right\|^{2}$. Define 
$a_k:=s(t_{k+1}-1)t_{k+1}$ and $\vu_k:=(t_{k+1}-1) (\vy_k-\vx_k)+\vy_{k}$.
Then the inequality \eqref{lyap decrease apg} and the definition \eqref{lyap: APG} yields that
\begin{equation*}
%\label{eq:estimatezk1}
\begin{array}{c}
     \frac{1}{2}\|\vu_k-\vx^\star\|^2
\leq
\rho^k
\cE_0-a_k\left(F\left(\vx_{k}\right)-F^{\star}\right)
\leq  \frac{\rho^k}{2}\left\|\vx_{0}-\vx^\star\right\|^{2}-\frac{\mu a_k}{2}\left\|\vx_{k}-\vx^\star\right\|^{2}, 
\end{array}
\end{equation*}
where $\rho = 1-{(1-Ls)\mu s}/{3}$.
Meanwhile, note that
$\vy_{k}=\frac{1}{t_{k+1}}\vu_{k}+\frac{t_{k+1}-1}{t_{k+1}}\vx_{k}$, one has
\begin{equation}
\label{apg ykxkx0}
\begin{array}{rl}
     \left\|\vy_{k}-\vx^\star\right\|^2
&\le
\big(1-\frac{1}{t_{k+1}}\big )
 \|\vx_{k}-\vx^\star \|^2
 +\frac{1}{t_{k+1}}\left\|\vu_{k}-\vx^\star\right\|^2
\\
&\le
\big(1-\frac{1}{t_{k+1}}
-\mu s (t_{k+1}-1)\big)
\left\|\vx_{k}-\vx^\star\right\|^2
+\frac{\rho^{k}}{t_{k+1}}
{\left\|\vx_{0}-\vx^\star\right\|^{2}}\\
&\le
\big[\left(1-\mu s\right)
\rho^{k-1}\max \big\{0,
1-\frac{1}{t_{k+1}}-\mu s(t_{k+1}-1)\big\}
+\frac{\rho^{k}}{t_{k+1}} \big]
\left\|\vx_0-\vx^\star\right\|^2.
\end{array}
\end{equation}
Taking \eqref{PG descent} and $\rho> 1-\mu s$ into account yields that \eqref{apg xkx0} holds for $k+1$.
% $$
% \begin{array}{l}
%      \|\vx_{k+1}-\vx^\star\|^2\leq (1-\mu s)\left\|\vy_k-\vx^\star\right\|^2\\
% \leq (1-\mu s) \max \big\{0,
% 1-\frac{1}{t_{k+1}}-\mu s(t_{k+1}-1)\big\}
% \left\|\vx_k-\vx^\star\right\|^2
% +\frac{1-\mu s}{t_{k+1}} \rho^{k}
% \left\|\vx_{0}-\vx^\star\right\|^{2}\\
% \leq (1-\mu s)\max \big\{0,
% 1-\frac{1}{t_{k+1}}-\mu s(t_{k+1}-1)\big\}
% \left(1-\mu s\right)
% \rho^{k-1}
% \left\|\vx_{0}-\vx^\star\right\|^{2}
% +\frac{1-\mu s}{t_{k+1}}
% \rho^{k}
% \left\|\vx_{0}-\vx^\star\right\|^{2}
% % \\
% % \leq (1-\mu s)\max \big\{\frac{1}{t_{k+1}},
% % 1-\mu s(t_{k+1}-1)\big\}
% % \rho^{k}
% % \left\|\vx_{0}-\vx^\star\right\|^{2}
% ,
% \end{array}
% $$
% so that \eqref{apg xkx0} holds for $k+1$ since $\rho=1- {(1-L s)\mu s}/{3}> 1-\mu s$.
Moreover, when $k$ is sufficiently large, 
it holds $\max \big\{\frac{1}{t_{k}}, 
1-\mu s(t_{k}-1)\big\}=\frac{1}{t_k}$, so we have
$\|\vx_k-\vx^\star\|^2
\leq \mathcal{O}(\rho^{k-1}/t_k)$. 
The proof is completed since $t_k=\mathcal{O}(k)$ when $t_k$ satisfies \eqref{step:NAG beta} or \eqref{step:NAG r}.
\end{proof}

\begin{remark}
    One also has $\|\vx_0-\vx^\star\|^2\le
\frac{ \rho^k \|\vx_0-\vx^\star\|^2 }{\mu s(t_{k+1}-1)t_{k+1}} $ from the inequality \eqref{apg linear} and strong convexity of $F$. This $\mathcal{O}({\rho}^{k-1}/k^2)$ rate is better than that in \eqref{thm: apgseq} when $k$ is large enough, but worse when $k$ is small since the coefficient is large. Moreover, the third iteration of \eqref{APG} shows that $\vy_{k+1}\in\conv\{\vx_{k+1},\vx_k\}$, which implies the globally R-linear convergence of $\{\vy_k\}$.
\end{remark}

\subsection{Convergence rate analysis of APG with gradient restarting}
\label{sec:apgc}
This subsection analyzes the gradient restarted APG method given by \Cref{al: reapgc}. 

\begin{algorithm}[ht]
\caption{Gradient restarted APG method for problem \eqref{obj: nonsmooth}}
\label{al: reapgc}
\begin{algorithmic}
\Require
Initial point $\vx_0=\vy_0$, step size $s>0$, $j=1$ and parameter sequence $\{t_k\}$ satisfying \eqref{NestRule}.  
\Ensure The minimizing sequences $\{\vx_k\}$ and $\{\vy_k\}$.
\For {$k=1,2,\ldots$}{
\State  $\vz_{k} \leftarrow \prox_{s g}(\vy_{k-1}-s\nabla f(\vy_{k-1}))$
\State $\vy_{k} \leftarrow \vz_{k}+\frac{t_{j}-1}{t_{j+1}}\left(\vz_{k}-\vx_{k-1}\right)$
\If{$\langle \vz_k-\vx_{k-1},\vy_{k-1}-\vz_k\rangle>0$}
\State  $\vx_k\leftarrow \prox_{s g}(\vx_{k-1}-s\nabla f(\vx_{k-1}))$
\State $\vy_k\leftarrow \vx_k,\quad j \leftarrow 1$
\Else
\State $\vx_k\leftarrow \vz_k,\quad j \leftarrow j+1$

\EndIf
}
\EndFor
\end{algorithmic}
\end{algorithm}

We have the following result regarding the global R-linear convergence of \Cref{al: reapgc}. 
\begin{theorem}\label{thm: reapgc}
Under Assumption \ref{ass:blanket}, let $\{\vx_k\}$ and $\{\vy_k\}$ be the sequences generated by \Cref{al: reapgc} with $s< 1/L$.
Then, it holds for all $k\ge 1$ that %one has
\begin{equation}\label{linear-apg-restart}
\begin{array}{l}
     \|\vx_k-\vx^\star\|^2\leq  (1-\mu s)\rho^{k-1}\|\vx_0-\vx^\star\|^2 \quad\mbox{with} \quad  \rho = 1- {(1-L s)\mu s}/{3}. 
\end{array}
\end{equation}
%with $\rho=1- \frac{(1-L s)\mu s}{3} \in \left(1-\mu s,1\right)$.
% Specifically, the constants $C$ and $\rho$ in \eqref{linear-apg-restart} satisfy
% \begin{equation}
% \label{choice-C-rho}
%     C=\frac{1-\mu s}{\rho},\quad\mbox{and}\quad \rho=1- \frac{(1-L s)\mu s}{3}.
% \end{equation}
%where $\rho$ is given in Proposition~\ref{thm: apgseq}.
Moreover, the iteration sequence $\{\vy_k\}$ and the function value sequence $\{F(\vx_k)\}$ also converge R-linearly with the same rate.
\end{theorem}

Before providing the proof, we first introduce some definitions. The first restart time of \Cref{al: reapgc} with initial point $\vx_0$ is denoted by
$$
\begin{array}{l}
     K(\vx_0): =\max\left\{ k\in\mathbb{N} \mid \left\langle \vy_{l-1}-\vx_{l},\vx_{l}-\vx_{l-1}  \right\rangle\leq 0\mbox{ for all } l\in [k]\right\}. 
\end{array}
$$
For the infinite sequence $\{\vx_i\}_{i=0}^{\infty}$ generated by the Algorithm \ref{al: reapgc}, define the number of iterations during the $i$-th restart interval by $K_i$ and use $S_i$ to denote the total number of iterations before the $i$-th restarting.
Then, it holds that 
\begin{equation}\label{def: Ki}
\begin{cases}
K_i=K(\vx_{S_{i-1}},f),\\
S_i=S_{i-1}+K_i
\end{cases}
    \quad \mbox{for}\quad i=1,2,\ldots
\end{equation}
with the initial value $S_0=K_0=0$. Moreover, given a $k\ge1$, let $m_k$ be the number of restarts before the $k$th step, that is
\begin{equation}\label{def: m_k}
\begin{array}{l}
     m_k=\max\{i|S_i<k\}.
\end{array}  
\end{equation}

\begin{proof}[\textbf{Proof of Theorem~\ref{thm: reapgc}}]

Given $k\ge 1$, according to the definition of $\{K_i\},\{S_i\}$ \eqref{def: Ki} and $m_k$ \eqref{def: m_k}, we have $k\in(S_{m_k},S_{m_{k}+1}]$ and $k=S_{m_k}+j$ for some $j\in[K_{m_k+1}]$. It is from \Cref{thm: apgseq} and $t_k\ge1 $ that
\begin{equation*}
\begin{array}{l}
     \|\vx_k-\vx^\star\|^2
     \leq (1-\mu s)\rho^{j-1} \|\vx_{S_{m_k}}-\vx^\star\|^2
     \le \cdots
\leq(1-\mu s)^{m_k+1} \rho^{j-1-m_k+\sum_{i=1}^{m_k} K_i}\left\|\vx_{0}-\vx^\star\right\|^{2},%\leq\frac{1-\mu s}{\rho}\rho^{k}\left\|\vx_{0}-\vx^\star\right\|^{2},
\end{array}
\end{equation*}
which implies \eqref{linear-apg-restart} by $k=j+\sum_{i=1}^{m_k} K_i$ and $\rho=1- (1-L s)\mu s/{3}>1-\mu s$.
Note that $\vy_k=\vx_k$ when restarting occurs and $\vy_k=\vx_{k}+\frac{t_{j}-1}{t_{j+1}}\left(\vx_{k}-\vx_{k-1}\right)$ for non-restart steps. Therefore, we know that $\vy_k\in\conv\{\vx_{k},\vx_{k-1}\}$, which implies 
$$
\begin{array}{l}
     \left\|\vy_k-\vx^\star\right\|^2\le \max\{\left\|\vx_k-\vx^\star\right\|^2,\left\|\vx_{k-1}-\vx^\star\right\|^2\}\le (1-\mu s)\rho^{k-2}\|\vx_0-\vx^\star\|^2\quad\mbox{for all }k\ge2.
\end{array}
$$
Moreover, according to \eqref{ykxstar} and $s < 1 / L$, one has for all $k\ge2$ that
$$
\begin{array}{l}
F\left(\vx_{k+1}\right)-F(\vx^\star)
\leq G_s(\vy_k)^T(\vy_k-\vx^\star)-\frac{s}{2}\left\|G_s(\vy_k)\right\|^2\le \frac{1}{2s}\left\|\vy_k-\vx^\star\right\|^2.
\end{array}
$$
Consequently, the global R-linear convergence of  $\{\vx_k\}$, $\{\vy_k\}$ and $\{F(\vx_k)\}$ are constructed.

\end{proof}

\begin{remark}
In fact, \Cref{thm: reapgc} establishs the global R-linear convergence for all restarted/non-restarted APG methods with $s\in(0,1/L)$ under Assumption \ref{ass:blanket}, including the gradient-based \Cref{al: reapgc}. 
\end{remark}

One can see that the parameter $\rho$ in \Cref{thm: reapgc} is the same as the one in \Cref{thm: APG} for the original APG method. 
When a uniform upper-bound for the restart intervals $\{K_i\}_{i=1}^{\infty}$ is available, one can obtain a smaller factor for the convergence rate, as stated in the following proposition.

\begin{proposition}
    \label{coro: bouned reapg}
Under Assumption \ref{ass:blanket}, let $\{\vx_k\}$ and $\{\vy_k\}$ be the sequences generated by \Cref{al: reapgc} with $s< 1/L$.
Assume the restart intervals $\{K_i\}$, defined by \eqref{def: Ki}, are uniformly bounded from above by some integer $\overline{\cal K}\ge2$.
Then, there exists a positive constant $\hat\rho$ such that
\begin{equation}
\label{coro-1-results}
    \left\|\vx_{k}-\vx^\star\right\|^{2}
\leq
(1-\mu s)\hat \rho^{k-1} \|\vx_0-\vx^\star\|^2
\quad \mbox{for all }k\ge 1,
\end{equation}
where $0<\hat\rho<\rho:=1- \frac{(1-L s)\mu s}{3}$. Moreover, the sequences $\{\vy_k\}$ and $\{F(\vx_k)\}$ also converge R-linearly with the same rate.
\end{proposition}

\begin{proof} 
The fact $x_0=y_0$ and $ x_1=y_1$ imply that $\left\langle \vy_{0}-\vx_{1},\vx_{1}-\vx_{0}  \right\rangle\leq 0$ and  
$ \left\langle \vy_{1}-\vx_{2},\vx_{2}-\vx_{1} \right\rangle\leq 0$, respectively. Then we have $K_i\ge 2, i\ge1$, with $K_i$ defined by \eqref{def: Ki}. Define the constant $\vartheta_{\overline{\cal K}}$ as
\begin{equation}
       \label{def: vartheta}
       \begin{array}{c}
             \vartheta_{\overline{\cal K}}:=\max\limits_{l=2,3,\ldots,{\overline{\cal K}}}
\left(\bar\theta_l\right)^{1/l}\quad\mbox{with}\quad \bar\theta_l:=\max\left\{ \frac{1}{t_{l}}, 1-\mu s (t_{l}-1) \right\}, l\ge1.
       \end{array}
   \end{equation}
Since $t_l>1$ for any $l\geq 2$, we have $\vartheta_{\overline{\cal K}}\in(0,1)$.
Note that the restart intervals are bounded as $2\le K_i\le \overline{\cal K}$ for $i\ge 1$, then one has $\bar\theta_{K_i}\leq (\vartheta_{\overline{\cal K}})^{K_i}$ and $\bar\theta_{j}\leq (\vartheta_{\overline{\cal K}})^{j-1}$.
According to \Cref{thm: apgseq}, we have
    \begin{equation*}%\label{reapgc descent}
    \begin{array}{ll}
\|\vx_k-\vx^\star\|^2
&
\leq (1-\mu s)\bar \theta_{j}
\rho^{j-1}
\|\vx_{S_{m_k}}-\vx^\star\|^2 \\[1mm]
&\leq \cdots%(1-\mu s)^2  \bar \theta_{j}\bar \theta_{K_{m_k}} \rho^{j+K_{m_k}-2} \|\vx_{S_{m_k-1}}-\vx^\star\|^2
\\[1mm]
&
\leq (1-\mu s)^{m_k+1}
\left(\bar\theta_j\prod_{i=1}^{m_k} \bar \theta_{K_i}
\right)\cdot
\rho^{\left(j-m_k-1+\sum\limits_{i=1}^{m_k} K_{i}\right)}
\|\vx_0-\vx^\star\|^2,
\\[1mm]
&
\leq (1-\mu s)^{m_k+1}\left(\vartheta_{\overline{\cal K}}\right)^{\left(j-1+\sum\limits_{i=1}^{m_k} K_{i}\right)}\rho^{k-m_k-1}
\|\vx_0-\vx^\star\|^2
\\[1mm]&
= (1-\mu s)^{m_k+1}\left(\vartheta_{\overline{\cal K}}\right)^{k-1}
\rho^{k-m_k-1}
\|\vx_0-\vx^\star\|^2,\\[1mm]&
\leq (1-\mu s)\hat \rho^{k-1}
\|\vx_0-\vx^\star\|^2,
\end{array}
\end{equation*}
where $\hat \rho:=\vartheta_{\overline{\cal K}}\rho<\rho$ and the last inequality follows from $\rho\ge1-\mu s$. The proof is completed with a similar analysis of $\vy_k$ and $F(\vx_k)$ in \Cref{thm: reapgc}.
\end{proof}

\begin{remark} 
As discussed in \Cref{remark tk<k}, we have $t_k<k$ for all $k\ge1$. Taking $\hat\rho<\rho$ into account, then there exists some positive integer $\hat K$ such that 
$$
\begin{array}{c}
     \max\left\{ \frac{1}{t_k}
 , 1-\mu s (t_k-1) \right\}\ge \frac{1}{t_k}> \frac{1}{k}>\left(\frac{\hat\rho}{\rho}\right)^k
 \quad \mbox{for all}\quad k>\hat K.
\end{array}
 $$
Comparing \eqref{apg xkx0} with \eqref{coro-1-results}, we know the restarted version converges faster than the original APG method near the optimal point. 
\end{remark}

With uniformly bounded restart intervals, the linear rate of the restarted APG method improves from $\rho$ to $\hat{\rho}<\rho$ (see Corollary \ref{coro: bouned reapg}). 
In addition, when the non-smooth term $g$ is vacuous, we can get a stronger result in the sense that $\|\vx_k-\vx^*\|\leq \mathcal{O}(\prod_{i=1}^{k-1}\eta_i)$ where $\{\eta_i\}_{i=1}^k$ is a monotone increasing sequence and its upper bound $\bar\eta$ is smaller than $\rho$ (See Theorem \ref{thm: renagc} in Appendix). Moreover, given uniformly bounded restart intervals, the linear rate further improves to $\hat{\eta}$ where~$\hat{\eta}<\min\{\hat{\rho},\bar \eta\}$ (See Corollary \ref{coro: bouned} in Appendix). We summarize these results in Table~\ref{comparison:rate} and provide additional details in the Appendix.

\small
\begin{table}[ht]
\footnotesize\tabcolsep 10pt
\begin{center}
    \caption{\footnotesize
Comparison of linear convergence rate between the APG and gradient restarted APG for solving $\min f+g$ with $s\in(0,1/L)$. ``GR+APG" is the gradient restarted APG method; ``UBC" denotes the condition that means the restart intervals are uniformly bounded; $^\dagger$: we assume that $t_k$ satisfies \eqref{step:NAG beta} or \eqref{step:NAG r}.}\label{comparison:rate}\vspace{-2mm}
\end{center}
\begin{center}
\begin{tabular}{cccccc}
\toprule
\multicolumn{1}{c|}{Objective function}   & \multicolumn{3}{l|}{\qquad\quad$f\in \mathcal{S}_{\mu,L},g\in\Gamma_0$}  & \multicolumn{2}{c}{$f\in \mathcal{S}_{\mu,L},g\equiv0$}  \\ \hline
\multicolumn{1}{c|}{\multirow{2}{*}{Algorithm}} & \multicolumn{1}{c|}{\multirow{2}{*}{APG}}  & \multicolumn{2}{c|}{GR+APG}  & \multicolumn{2}{c}{GR+APG}                                       \\ \cline{3-6}
\multicolumn{1}{c|}{}  & \multicolumn{1}{c|}{}   & \multicolumn{1}{c|}{Original}  & \multicolumn{1}{c|}{+UBC}  & \multicolumn{1}{c|}{Original}  & \multicolumn{1}{c}{+UBC} \\ \hline
\multicolumn{1}{c|}{$\|\vx^k-\vx^*\|$}    & \multicolumn{1}{c|}{$\mathcal{O}(\rho^{k-1}/k)^\dagger$} & \multicolumn{1}{c|}{$\mathcal{O}(\rho^{k})$} & \multicolumn{1}{c|}{$\mathcal{O}(\hat{\rho}^k), \hat{\rho}<\rho$} & \multicolumn{1}{c|}{$\mathcal{O}(\bar\eta),\bar\eta<\rho$} & \multicolumn{1}{c}{$\mathcal{O}(\hat{\eta}^k), \hat{\eta}<\min\{\hat{\rho},\bar \eta\}$} \\ \bottomrule
\end{tabular}

\end{center}
\end{table}

\normalsize

%The results in this section demonstrate improvements in the upper bound of $\|\vx_k-\vx^*\|$ for restart strategies with uniformly upper bound. 
The next section will provide a concrete example that definitively demonstrates the advantages of the gradient restarted scheme. Specifically, the gradient restarted version improves upon the optimal rate attained by the APG method for the given example.

\section{Continuous analysis of gradient restarting}
\label{sec:gradode}
In this section, motivated by the work of Su, Boyd, and Cand\'es \cite{su2016differential} for the continuous counterpart of the APG method when $g(x)\equiv0$, we study the continuous model of the gradient restarted APG under the same setting and show its provable advantage of gradient restarting.

When $g$ is vacuous, $f$ is a convex quadratic function, and $\{\beta_k\}$ satisfies \eqref{step:NAG beta} or \eqref{step:NAG r} with $r=2$, the ODE model \eqref{low-c} has the lower bound \eqref{eq:suupbound} for the objective value, which excludes the possibility of linear convergence for the solution trajectory. 
It is also proved in \cite{su2016differential} that the sublinear convergence rate $\mathcal{O}(1/t^3)$ is achievable, making it the optimal rate for the non-restarted version.
Moreover, it was shown in \cite{drori2017exact} that some quadratic function represents the worst-case scenario for first-order algorithms, in the sense of reaching the exact lower bound for first-order algorithms established in \cite{kim2016optimized} for $L$-smooth and convex objective functions. 
Therefore, even though the quadratic case is the simplest, 
it is desirable to know whether the corresponding restarted gradient version can overcome the limitation of the optimal $\mathcal{O}(1/t^3)$ rate.
In this case, since the ODE \eqref{low-c} is invariant under rotation and the linear term inf $f$ can be absorbed into the quadratic term, we only need to consider the more specific one that 
\begin{equation}
\label{obj: diag}
\begin{array}{c}
     f(\vx) = \frac{1}{2}\langle \vx,\boldsymbol \Lambda \vx\rangle,
\end{array}
\end{equation}
where $\boldsymbol \Lambda = \mathrm{diag}(\lambda_1,\lambda_2,\cdots,\lambda_n)$ and $\lambda_i>0, i\in[n]$. 

\subsection{Continuous model for gradient restarting}
Writing the initial point $\vx_0$ explicitly as $(x_{0,1},x_{0,2},\ldots,x_{0,n})$, the ODE \eqref{low-c} with $f$ being given by \eqref{obj: diag} admits the decomposition
\begin{equation}
\label{q_eq}
\left\{
\begin{array}{l}
\ddot{X}_{i}(t)+\frac{3}{t} \dot{X}_{i}(t)+\lambda_{i} X_{i}(t)=0,\ t>0,
\\[1mm]
X_{i}(0)=x_{0, i},\  \dot{X}_{i}(0)=0
\end{array}
\right.
\quad \mbox{for }
i\in [n].
\end{equation}
Let ${\vX}(t)$ be the corresponding solution. 
Without loss of generality, we assume that $\dot{\vX}(t)\neq \vx^\star$ for all $t>0$. 
Recall that, when $g(\vx)=0$, condition \eqref{restart:gradient} is equivalent to $\langle \nabla f(\vy_k), \vx_{k+1}-\vx_k\rangle >0$, whose continuous counterpart is given by \eqref{eq:gradrest}. 
Thus, we define the gradient restart time for the ODE \eqref{q_eq} by
\begin{equation}\label{grt}
\begin{array}{c}
     T^{\rm gr}(\vx_0;f)
=\sup\left\{t>0\mid \langle \nabla f(\vX(u)), \dot{\vX}(u)\rangle <0,\forall u\in (0,t)\right\}.
\end{array}
\end{equation}
Starting from $\vX(0)=\vx_0$,
$f(\vX(t))$ is monotonically decreasing along with $t\in(0,T^{\rm gr}(\vx_0;f))$ since
\begin{equation*}
\begin{array}{c}
     \frac{\mathrm{d}f(\vX(t))}{\mathrm{d}t}=\left\langle\nabla{f(\vX(t))},\dot{\vX}(t)\right\rangle<0,
\quad \forall t\in(0,T^{\rm gr}(\vx_0;f)). 
\end{array}
\end{equation*}
Following the terminology in \cite{su2016differential}, we get into the details of the ODE model \eqref{low-c} with the gradient restart scheme \eqref{grt}. Let $E_0=0$ and $\vr_0=\vx_0$, we recursively define two sequences
\begin{equation}
\label{def: Tixi}
\begin{array}{l}
     E_{i+1}=T^{\rm gr}(\vr_i;f)
\quad\mbox{and}\quad
\vr_{i+1}=\vY_{i+1}(E_{i+1}), 
\end{array}
\end{equation}
where $\vY_{i+1}(t) \mbox{ solves \eqref{low-c} with } \vx_0=\vr_i$.
Here, the sequence $\{E_i\}$ denotes the length of the restart interval, and the sequence $\{\vr_i\}$ denotes the restart point. Then, by denoting $\tau_i:=\sum_{j=0}^{i} E_j$, $i\ge 0$, the gradient restart scheme satisfies the following piecewise ODE:
\begin{equation}
\label{ODE3}
\left\{
\begin{array}{l}
     \ddot{\vX}(t)+\frac{3}{t-\tau_i}\dot{\vX}(t)+\nabla{f(\vX(t))}=0, \quad\mbox{for}\quad t\in\left( \tau_i, \tau_{i+1}  \right], \\[2mm]
     \vX(\tau_i)=\vr_i,\quad \dot{\vX}(\tau_i)=0
\end{array}
\right.
\quad \mbox{for } i=0,1,2,\cdots.
\end{equation}
Let $\vX^{\rm gr}(t)$ be the solution to \eqref{ODE3}.
The following results hold.
\begin{description}
\item[(a)] $\vX^{\mathrm{gr}}(t)$ is continuous for $t>0$, with $\vX^{\mathrm{gr}}(0)=\vx_0$.
\item[(b)] $\vX^{\mathrm{gr}}(t)$ satisfies the ODE \eqref{low-c} for $0<t<E_1:=T^{\rm gr}(\vx_0;f)$.
\item[(c)]  $\vr_{i}=\vX^{\rm gr}(\tau_i)$
and $E_{i+1}=T^{\rm gr}(\vr_{i};f)$
for $i\geq 1$.
\item[(d)]
$\vY_{i+1}(t)=\vX^{\rm gr}(\tau_i+t)$ satisfies the ODE~\eqref{low-c} with $\vY_{i+1}(0)=\vr_{i}$ for $0<t\leq E_{i+1}$.
\end{description}

\subsection{Breaking the limitation by gradient restarting}
In this subsection, we show the solution of \eqref{q_eq} has the global R-linear convergence in terms of $f(\vX^{\rm gr}(t))$, which breaks the limitation \eqref{eq:suupbound} of the continuous model \eqref{low-c} for the original APG under the same setting. 

Recall that the speed restart scheme \eqref{restart:speed}, introduced in \cite[Section 5.1]{su2016differential}, defines the restart time by
\begin{equation*}
%\label{sr time}
\begin{array}{c}
     T^{\rm sr}\left(\vx_{0} ; f\right)
:=\sup \left\{t>0\mid \frac{\mathrm{d}\|\dot{\vX}(u)\|^{2}}{\mathrm{~d} u}>0  \  \forall u \in(0, t)\right\}. 
\end{array}
\end{equation*}
According to \eqref{eq:grsr}, one has
$T^{\rm gr}(\vx_0;f)\geq T^{\rm sr}(\vx_0;f)$.
In \cite{su2016differential}, the authors established that when $f$ is fixed,
the speed restart interval $T^{\rm sr}(\vx_0;f)$ has uniform upper and lower bounds, which are independent of the initial point $\vx_0$. This result is critical for proving the global R-linear convergence of the speed restarted method.
By \cite[Lemma 25]{su2016differential}, we know
\begin{equation}
    \label{restart-time:lower bound}
    \begin{array}{c}
         T^{\rm gr}(\vx_0;f)\geq T^{\rm sr}(\vx_0;f)\geq \frac{4}{5\sqrt{L}}.
    \end{array}
\end{equation}
Thus, our main effort here is to prove the uniformly upper bound for $T^{\rm gr}(\vx_0;f)$ with given $f$ to establish the linear convergence. 
We first have the following result. 
\begin{proposition}
\label{thm: Ft}
Assume $f$ is defined as in \eqref{obj: diag}, and let $X_i(t)$, for $i\in [n]$, be the solution of \eqref{q_eq}. There exists a positive constant $T$ such that, for any initial point $\vx_0$, a time point $t_{\vx_0}$ within the interval $\left(0,T\right]$ exists and fulfills the condition that 
$\langle\nabla f(\vX(t_{\vx_0})), \dot{\vX}(t_{\vx_0})\rangle\ge 0$. 
\end{proposition}

The proof of \Cref{thm: Ft} is technical and long, so we postpone it to the next subsection.
Based on \Cref{thm: Ft}, it is easy to obtain the following result.

\begin{corollary}\label{coro: quad}
Assume $f=\frac{1}{2}\langle \vx, \vA\vx\rangle +\langle \vb, \vx\rangle$ with a symmetric positive semi-definite $\vA$ and $\vb\in\operatorname{Range}(\vA)$, and let $X(t)$ be the solution of \eqref{low-c}. Then there exists a positive constant $T$ such that, for any initial point $\vx_0$, a time point $t_{\vx_0}$ within the interval $\left(0,T\right]$ exists and fulfills the condition:
$$
\begin{array}{c}
     \langle\nabla f(\vX(t_{\vx_0})), \dot{\vX}(t_{\vx_0})\rangle\ge 0. 
\end{array}
$$
\end{corollary}

\begin{proof}
The assumption of the function $f$ indicates that there exists a minimizer $\vx^\star\in\RR^n$ of $f$.
Let $\vA=\vQ^T\boldsymbol \Lambda \vQ$  be the spectral decomposition of $\vA$ where $\vQ$ is orthonormal and $\boldsymbol \Lambda = \mathrm{diag}(\lambda_1,\lambda_2,\cdots,\lambda_n)$ containing all the eigenvalues of $\vA$. Define $\hat f : \mathbb{R}^n\to \mathbb{R}$ as $\hat{f}(\vy):= \frac{1}{2}\langle \vy,\boldsymbol\Lambda \vy\rangle$, it is easy to know that $\vX(t)$ solves \eqref{low-c} with $\vx_0$ and $f$ is equivalent to $\vY(t):=\vQ(\vX(t)-\vx^\star)$ solves \eqref{low-c} with $\vy_0:=\vQ(\vx_0-\vx^\star)$ and $\hat f$. Moreover, we have
$$
\begin{array}{l}
     \left\langle \nabla f(\vX(t)), \dot \vX(t)\right\rangle=\langle \boldsymbol
 \Lambda \vQ(\vX(t)-\vx^\star), \vQ\dot \vX(t)\rangle
   =\left\langle \nabla \hat{f}(\vY(t)), \dot \vY(t)\right\rangle.
\end{array}
$$
Note that $\hat{f}(\vy)=\sum_{i=1}^n \lambda_i y_i^2=\sum_{i\in\mathcal{I}}^n \lambda_i y_i^2$ with $\mathcal{I}:=\{i\in[n] \mid \lambda_i>0\}$. Then, without loss of generality, we assume $\lambda_i>0. i\in[n]$. Thus, according to \Cref{thm: Ft}, the restart time of $\vX(t)$ is upper bounded by $T$ since $T^{\rm gr}(\vx_0;f)=T^{\rm gr}(\vQ(\vx_0-\vx^\star);\hat f)\leq T$. 
\end{proof}
\color{black}

From the above corollary, we can get the following result regarding the convergence rate of $f(X^{\rm gr}(t))$, which constitutes the main result of this section. 
It is worth noting that the solution trajectory of \eqref{low-c} does not possess this property, as shown in \eqref{eq:suupbound}.

\begin{theorem}
\label{thm: linear quad}
Assume $f=\frac{1}{2}\langle \vx, \vA\vx\rangle +\langle \vb, \vx\rangle$ with a symmetric positive semi-definite $\vA$ and $\vb\in\operatorname{Range}(\vA)$, and let $\vX(t)$ be the solution of \eqref{low-c}.
Then, the function value $f(\vX^{\rm gr}(t))$ and the trajectory $\vX^{\rm gr}(t)$ converges to $f^\star$ and $\vx^\star$, respectively, at a globally R-linear rate of convergence.
\end{theorem}

\begin{proof}[The proof of \Cref{thm: linear quad}]
Let $\vX(t)$ be the solution of \eqref{low-c} with initial value $\vx_0$. According to \cite[Lemma 12]{su2016differential}, for any given $\vx_0$, there exists a universal constant $C\in (0,\frac{3}{25})$ such that
\begin{equation*}
\begin{array}{l}
     f\left(\vX(T^{\rm sr}(\vx_0;f))\right)-f^{\star} \leq\left(1-\frac{C \mu}{L}\right)\left(f\left(\vx_{0}\right)-f^{\star}\right). 
\end{array}
\end{equation*}
Since $\frac{\mathrm{d}f(\vX(t))}{\mathrm{d}t}\leq 0$ before the gradient restart time
and  $T^{\rm gr}(\vx_0;f)\geq T^{\rm sr}(\vx_0;f)$ by \eqref{eq:grsr}, one has that from the same initial point $\vx_0$,
\begin{equation}\label{lbdescent}
\begin{array}{l}
     f(\vX(T^{\rm gr}(\vx_0;f)))-f^{\star}\leq f(\vX(T^{\rm sr}(\vx_0;f)))-f^{\star} \leq\left(1-\frac{C \mu}{L}\right)\left(f\left(\vx_{0}\right)-f^{\star}\right). 
\end{array}
\end{equation}
Let $m$ denote the total number of gradient restarts before time $t$, we have $t\geq\tau_m$ and $m \geq \lfloor{t/T}\rfloor \geq t/T-1$ according to \Cref{coro: quad}. Note that
$$
\begin{array}{l}
     \vX^{\mathrm{gr}}(\tau_{i+1})=\vX^{\mathrm{gr}}(\tau_{i}+E_{i+1})=\vY_{i+1}(E_{i+1})=\vY_{i+1}(T^{\rm gr}(\vr_i;f)) \quad
\mbox{for } \quad i=0,1,\ldots,m-1,
\end{array}
$$
where $\tau_{i},E_{i},\vY_{i}$ and $\vr_i$ are defined in \eqref{def: Tixi}. Thus, the inequality \eqref{lbdescent} implies that
$$
\begin{array}{rl}
f\left( \vX^{\mathrm{gr}}(\tau_{i+1})\right)-f^\star
&=f\left( \vY_{i+1}(T^{\rm gr}(\vr_i;f))\right)-f^\star\\
&\leq \left(1-\frac{C \mu}{L}\right)\left(f\left(\vr_i\right)-f^{\star}\right)
=\left(1-\frac{C \mu}{L}\right)\left(f\left(\vX^{\mathrm{gr}}(\tau_{i})\right)-f^{\star}\right),
\end{array}
$$
since $\vY_{i+1}$ solves \eqref{low-c} with initial value $\vr_i$.
Therefore, we obtain the linear convergence by
\begin{equation*}
	    \begin{array}{rl}
	    	f\left(\vX^{\mathrm{gr}}(t)\right)-f^{\star} & \leq f\left(\vX^{\mathrm{gr}}\left(\tau_m\right)\right)-f^{\star}
	    	 \leq \big(1-\frac{C\mu}{L}\big)(f\left(\vX^{\rm gr}\left(\tau_{m-1}\right)\right)-f^{\star}) 
             \\[1mm]
	    	& \leq \big(1-\frac{C\mu}{L}\big)^m(f(\vx_0)-f^{\star}) \leq e^{-C\mu m/L}(f(\vx_0)-f^{\star})
            \\[1mm]
	    	& \leq  c_1e^{-C\mu t/LT}\frac{L\|\vx_0-\vx^\star\|^2}{2} =\frac{c_1 L\|\vx_0-\vx^\star\|^2}{2}e^{-c_2t},
	    \end{array}
	\end{equation*}
where $c_1 = e^{C\mu/L}$ and $c_2 = C\mu/LT$. The proof is completed with the strong convexity of $f$, which implies $f\left(\vX^{\mathrm{gr}}(t)\right)-f^{\star}\ge\mu\|\vX^{\rm gr}(t)-\vx^\star\|^2/2$.
\end{proof}

% {\color{magenta}
% [Rewrite to a concluding paragraph or remark.]
% This improvement from sublinear to linear convergence validates the advantages of the gradient restarting.
% Moreover, the result~\eqref{convgence-linear} partially addresses the open question of whether there exists a provable linear convergence rate for the gradient restart scheme \eqref{eq:gradrest} when applied to strongly convex problems \cite[pp. 27]{su2016differential}.Building upon these results, this work provides a theoretical understanding of the inherent properties of the gradient restarting within the accelerated gradient methods, which are consistently observed in numerical experiments.}

\begin{remark}
    \Cref{thm: linear quad} demonstrates that the gradient restarted ODE \eqref{ODE3} improves the convergence rate from the sublinear rate \eqref{eq:suupbound} of the original ODE \eqref{low-c} to linear when minimizing a quadratic and convex function. This enhancement theoretically validates the advantages of the gradient restarting. 
\end{remark}

%and partially addresses the open question of whether there exists a provable linear convergence rate for the gradient restart scheme \eqref{eq:gradrest} when applied to strongly convex problems \cite[pp. 27]{su2016differential}.

\subsection{Proof of Proposition \ref{thm: Ft}}
Let $\vX(t)$ be the solution of \eqref{q_eq}. To prove Proposition \ref{thm: Ft}, we need to consider the property of the function $S(t):\RR_+\rightarrow\RR$ defined by
\begin{equation}\label{Gt}
\begin{array}{l}
     S(t):=-\frac{\pi t^3}{4}\langle\nabla f(\vX(t)), \dot{\vX}(t)\rangle=-\frac{\pi t^3}{4}\sum_{i=1}^{n} \lambda_{i} X_{i}(t) \dot{X}_{i}(t), \ t>0. 
\end{array}
\end{equation}
Note that $\langle\nabla f(\vX(t)), \dot{\vX}(t)\rangle\ge0$ is equivalent to $S(t)\le 0$. Then the inequality~\eqref{restart-time:lower bound} yields that $S(t)>0$ for all $t\in(0,4/5\sqrt{L})$. Next, using the Bessel function, we can get the explicit expression of $S(t)$ by the following lemma.
\begin{lemma}%\label{lemma_bes}
Let $X_i$ to be the solution of \eqref{q_eq} for all $i\in[n]$.
One has
% $$
% \begin{array}{l}
%      H(t) = \sum_{i=1}^n H_i(t)
% \quad\mbox{with}\quad
% H_i(t)=-\frac{4\sqrt{\lambda_{i}}x^2_{0,i}}{t^2}J_1(\sqrt{\lambda_{i}}t)J_2(\sqrt{\lambda_{i}}t), \quad i\in[n],
% \end{array}
% $$
\begin{equation}\label{g}
\begin{array}{l}
     S(t)=\sum^{n}_{i=1}x^2_{0,i}G(\sqrt{\lambda_{i}}t)
\quad\mbox{with}\quad
G(u)=\pi u J_1(u)J_2(u),
\end{array}
\end{equation}
where $J_{q}(u)=\sum_{p=0}^{\infty} \frac{(-1)^{p}}{(p)!(p+q)!} \left(\frac{u}{2}\right)^{2 p+q}$ is Bessel function of the first kind with order $q$.
\end{lemma}
\begin{proof}
According to \cite[ Section 3.2]{su2016differential}, the solution to \eqref{q_eq} is given by
\begin{equation}\label{X_i}
\begin{array}{l}
     X_{i}(t)=\frac{2 x_{0, i}}{t \sqrt{\lambda_{i}}} J_{1}\left( \sqrt{\lambda_{i}}t\right), 
\end{array}
\end{equation}
where $J_1$ is the Bessel function of the first kind of order one.
Utilizing the iterative formula for the Bessel function \cite{watson1995treatise},
\begin{equation*}
\begin{array}{l}
     \frac{\mathrm{d}}{\mathrm{d}u}\left[u^{-q} J_{q}(u)\right]=-u^{-q} J_{q+1}(u), 
\end{array}
\end{equation*}
and combining it with \eqref{X_i}, we derive
\begin{equation*}
\begin{array}{l}
     \dot{X}_i(t)=-\frac{2x_{0,i}}{t}J_2(\sqrt{\lambda_{i}}t). 
\end{array}
\end{equation*}
The proof is completed with
$S(t)=-\frac{\pi t^3}{4}\sum_{i=1}^{n} \lambda_{i} X_{i}(t) \dot{X}_{i}(t)$.
\end{proof}

% Next, we define the function
% \begin{equation}\label{g}
% G(u)=\pi u J_1(u)J_2(u),
% \end{equation}
% so that, by \Cref{lemma_bes}, $H(t)$ can be expressed as
% \begin{equation}\label{F}
% \begin{array}{l}
%      H(t)=
% -\sum_{i=1}^n
% \frac{4\sqrt{\lambda_{i}}x^2_{0,i}}{t^2}J_1(\sqrt{\lambda_{i}}t)J_2(\sqrt{\lambda_{i}}t)
% =
% -\sum\limits^{n}_{i=1}\frac{4x^2_{0,i}}{\pi t^3}G(\sqrt{\lambda_{i}}t).
% \end{array}
% \end{equation}
Then we characterize the asymptotic behavior of $G(u)$.
\begin{lemma}
\label{lemma_as}
Let $G$ be defined in \eqref{g}. Given an $\epsilon>0$, there exists a constant $T_{\epsilon}>0$ such that
$$
\begin{array}{l}
     |G(u)-\cos(2u)|\leq \epsilon,\quad\forall u>T_\epsilon. 
\end{array}
 $$
\end{lemma}
\begin{proof}
It follows from \cite{watson1995treatise} that the asymptotic expansions of
$J_{q}(u)$ is given by
$$
\begin{array}{l}
     J_{q}(u) =\sqrt{\frac{2}{\pi u}} \cos \left(u-\frac{\pi}{4}-\frac{q \pi}{2}\right)+o\left(\frac{1}{\sqrt{u}}\right).
\end{array}
$$
Therefore, we can get
\begin{equation*}
\begin{array}{rl}
J_1(u)J_2(u) & =\frac{2}{\pi u }\cos\left(u-\frac{\pi}{4}-\frac{\pi}{2}\right)\cos\left(u-\frac{\pi}{4}-\pi\right) + o\left(\frac{1}{u}\right) \\[1mm]
& =-\frac{1}{\pi u }\sin(2u-\frac{\pi}{2}) + o\left(\frac{1}{u}\right) =\frac{1}{\pi u }\cos(2u) + o\left(\frac{1}{u}\right).
\end{array}
\end{equation*}
Therefore, one gets
$\lim\limits_{u\to\infty}|G(u)-\cos(2u)|=0$, which completes the proof.
\end{proof}

Lemma \ref{lemma_as} describes the approximation of $G(u)$ by the cosine function. Therefore, the remaining issue is determining when the weighted sum of the cosine functions becomes nonpositive. We claim the following lemmas. 
\begin{lemma}[Dirichlet's Theorem on Simultaneous Approximation, \cite{schmidt1980diophantine} Chapter 2, Theorem 1A]\label{thm: diri}
    Suppose that $\alpha_1, \ldots, \alpha_n$ are $n$ real numbers and that $Q>1$ is an integer. Then there exist integers $q, p_1, \ldots, p_n$ such that
$$
\begin{array}{l}
     1 \leq q<Q^n \quad \text { and } \quad\left|\alpha_i q-p_i\right| \leq 1/Q \quad \mbox{for all} \quad i\in[n].
\end{array}
$$
\end{lemma}
\begin{lemma}\label{lemma:epin}
    Define $\Omega_1^n:=[-\frac{T_1}{6},\frac{T_1}{6}]\times \ldots\times [-\frac{T_n}{6},\frac{T_n}{6}]$ and assume $\{\lambda_i\}_{i=1}^n$ be a set of positive real numbers. Then there exist positive constants $\epsilon_n$ and $\bar t_{max}$ such that, for any point $\vx \neq 0$ and $\omega:=[\omega_1,\ldots,\omega_n]\in\Omega_1^n$, a time $\bar t$ within the interval $(0, \bar t_{max}]$ exists and fulfills the condition:
        $$
        \begin{array}{l}
             \sum_{i=1}^{n} x^2_{i} \cos \left(2 \sqrt{\lambda_{i}}( \bar t+\omega_i)\right)\le -\|\vx\|^2\epsilon_n. 
        \end{array}
    $$

\end{lemma}

\begin{proof}%[Proof of \Cref{lemma:epin}] 
The results trivially hold when $n=1$, thus we only need to figure out the scenario $n\ge2$. 
Given $\vx\in\mathbb{S}^{n-1}:=\{\vx\in\RR^n \mid \|\vx\|_2=1\}$ and $\omega\in\Omega_1^n$, we first claim that 
\begin{equation}\label{claim}
    \begin{array}{l}
         \mbox{there exists a positive number}\quad t_{\vx}^{\omega} \quad\mbox{such that}\quad\sum_{i=1}^{n} x^2_{i} \cos \left(2 \sqrt{\lambda_{i}}( t_{\vx}^{\omega}+\omega_i)\right)<0.
    \end{array}
\end{equation}
Since $\cos \left(2 \sqrt{\lambda_{i}}( t+\omega_i)\right)$ is continuous w.r.t $t$ and $\cos \left(2 \sqrt{\lambda_{i}}( \omega_i)\right)>0$ for all $i\in[n]$, then there exists a $T_{low}\in(0,T_1)$ such that 
\begin{equation*}
\begin{array}{l}
     \delta:=\int_{0}^{T_{low}} \sum_{i=1}^{n} x^2_{i} \cos \left(2 \sqrt{\lambda_{i}}( t_{\vx}^{\omega}+\omega_i)\right) \mathrm{d} t>0.
\end{array}
\end{equation*}
Assume $T_{max}=\max_{i\in[n]} T_i$ and apply \Cref{thm: diri} to the positive integer $Q:=\left\lfloor T_{max} /{\delta  }\right\rfloor+1$ and $n-1$ real numbers $T_1/ T_2,\ldots T_1/T_n$, then there exist integers $q, p_1, \ldots, p_n$ such that 
$$
\begin{array}{l}
     1 \leq q<Q^{n-1} \quad \text { and } \quad\left|(T_1q)/{T_i} -p_i\right| \leq 1/{Q} \quad \mbox{for all } i\in[n-1]+1.
\end{array}
$$
It yields $|q T_1 -p_i T_i|\le T_i/{Q}$ for all $i\in[n-1]+1$.
Let $T_{upp}:=qT_1$, then we have 
\begin{equation*}
    \begin{array}{rl}
    \int_{0}^{T_{upp}} \sum_{i=1}^{n} x^2_{i} \cos \left(2 \sqrt{\lambda_{i}}( t+\omega_i)\right) \mathrm{d} t 
    &\le \|\vx\|_2^2 \max_{i\in[n]} \int_{0}^{qT_1}\left|\cos \left(2 \sqrt{\lambda_{i}} (t+\omega_i)\right)\right|\mathrm{d} t \\[1mm]
    &\le  \max_{i\in[n-1]+1} \int_{0}^{|q T_1 -p_i T_i|}\left|\cos \left(2 \sqrt{\lambda_{i}} (t+\omega_i)\right)\right|\mathrm{d} t \\[1mm]
    &\le \max_{i\in[n-1]+1}|q T_1 -p_i T_i|\le T_{max}/{Q}<\delta,
\end{array}
\end{equation*}
where the first inequality above comes from $\|\vx\|_2^2=1$ and $|\cos(\cdot)|\le1$ and the third inequality is from $|\cos(\cdot)|\le1$.
Therefore, the claim \eqref{claim} is proved by
$$
\begin{array}{l}
     \int_{T_{low}}^{T_{upp}} \sum_{i=1}^{n} x^2_{i} \cos \left(2 \sqrt{\lambda_{i}}( t+\omega_i)\right) \mathrm{d} t<\delta-\delta= 0,
\end{array}
$$
and $T_{upp}=qT_1\ge T_1>T_{low}$.

Next, we eliminate the dependence of $t$ on $\vx_0$ with the Heine–Borel theorem. Define the smooth function $\varphi(\vx,\omega,t):\RR^n\times\RR^n\times\RR_{+}\rightarrow\RR$ by
\begin{equation*}
\begin{array}{l}
     \varphi(\vx,\omega,t):=\sum_{i=1}^{n} x_i^2\cos \left(2 \sqrt{\lambda_{i}} (t+\omega_i)\right), 
\end{array}
\end{equation*}
where $\omega=[\omega_1,\omega_2,\ldots,\omega_3]$. The claim \eqref{claim} and the continuity of $\varphi$ with respect to $\vx$ and $\omega$ implies that, there exists an open ball ${\mathbb B}_{\delta_{\vx}^{\omega}}(\vx,\omega) := \{(\vx',\omega')\in\RR^{2n} \mid \|\vx-\vx'\|_2^2+\|\omega-\omega'\|_2^2<\delta_{\vx}^{\omega}\}$ such that $\varphi(\vx',\omega', t_{\vx}^{\omega})<\frac{1}{2}\varphi(\vx,\omega, t_{\vx}^{\omega})<0$ for all $(x',\omega')\in {\mathbb B}_{\delta_{\vx}^{\omega}}(\vx,\omega)$.

% $\delta_{\vx}^{\omega}$ such that $\varphi(\vx',\omega', t_{\vx}^{\omega})<\frac{1}{2}\varphi(\vx,\omega, t_{\vx}^{\omega})<0$ for all $(x',\omega')\in {\mathbb B}(\vx,\delta_{\vx}^{\omega})\times {\mathbb B}(\omega,\delta_{\vx}^{\omega})$. 

% an open ball ${\mathbb B}(\vx,\delta_{\vx}) = \{\vx'\in\RR^n|\|\vx-\vx'\|_2^2+\|\vx-\vx'\|_2^2<\delta_{\vx}\}$ such that 

Due to the compactness of $\mathbb{S}^{n-1}\times \Omega_1^n$ and  ${\mathbb S}^{n-1}\times \Omega_1^n \subset \cup_{\vx\in {\mathbb S}^{n-1}\times \Omega_1^n}{\mathbb B}_{\delta_{\vx}^{\omega}}(\vx,\omega)$, there exists a finite set $\{(\vx^{(j)},\omega^{(j)})\}_{j=1}^M\subset {\mathbb S}^{n-1}\times \Omega_1^n$ such that $\cup_{j=1}^M {\mathbb B}_{\delta^{(j)}}(\vx^{(j)},\omega^{(j)})\supset {\mathbb S}^{n-1}\times \Omega_1^n$ with $\delta^{(j)}=\delta_{\vx^{(j)}}^{\omega^{(j)}}$.
Define
$$
\begin{array}{l}
     \epsilon_n := -\frac{1}{2}\max_{j\in[M]}\{\varphi(\vx^{(j)},\omega^{(j)},\bar{t}_j)\}>0
\quad\mbox{and}\quad
\mathbb{M} :=\{\bar t_j\}_{j=0}^M \quad\mbox{with}\quad \bar t_j=t_{\vx^{(j)}}^{\omega^{(j)}} \mbox{ obtained from }\eqref{claim}.
\end{array}
$$
Note that $\frac{\vx}{\|\vx\|}\in {\mathbb S}^{n-1}$ for the given $\vx\neq 0$. Thus, there exists some index $i_0\in[M]$
such that $\left(\frac{\vx}{\|\vx\|},\omega \right)\in  {\mathbb B}_{\delta^{(i_0)}}(\vx^{(i_0)},\omega^{(i_0)})$ and $\bar{t}_{i_0}\in\mathbb{M}$.
Thus, we have
\begin{equation*}
\begin{array}{l}
     \varphi(\vx,\omega,\bar{t}_{i_0})= \|\vx\|^2\varphi\left(\frac{\vx}{\|\vx\|},\omega,{t}_{i_0}\right)
<\frac{\|\vx\|^2}{2} \varphi(\vx^{(i_0)},\omega^{(i_0)},{t}_{i_0})
\leq -\|\vx\|^2\epsilon_n, 
\end{array}
\end{equation*}
which completes the proof with $\bar t_{max}:=\max_{j\in[M]}\bar t_j$.
\end{proof}

Note that the numbers $\epsilon_n$ and $\bar t_{max}$ in Lemma~\ref{lemma:epin} are independent of the $\vx_0$. Combining the Lemma~\ref{lemma_as}, Proposition~\ref{thm: diri} and Lemma~\ref{lemma:epin}, we establish the proof of the  Proposition~\ref{thm: Ft}.

\begin{proof}[Proof of \Cref{thm: Ft}]
The results trivially hold when $\vx_0=0$. Assume that $\vx_0\neq 0$,  applying \cref{lemma_as} to the constant $\epsilon_n$ obtained in \Cref{lemma:epin}, there exists a constant $\hat T>0$ such that
\begin{equation}
    \label{approx}
    \begin{array}{l}
         |G(\sqrt{\lambda_i}t)-\cos(2\sqrt{\lambda_i}t)|\leq  \epsilon_n, \quad\forall t>\hat T,
\quad
\forall i\in[n].
    \end{array}
\end{equation}
Applying \Cref{thm: diri} to the positive integer $Q:=6\lfloor \hat T/T_1\rfloor+6$ and $n-1$ real numbers $T_1/ T_2,\ldots T_1/T_n$, then there exist integers $q, p_1, \ldots, p_n$ such that 
$$
\begin{array}{l}
     1 \leq q<Q^{n-1} \quad \text { and } \quad\left|(T_1q)/T_i-p_i\right| \leq 1/Q \quad \mbox{for all } i\in[n-1]+1. 
\end{array}
$$
Equivalently, we have, for all $i\in[n-1]+1$, 
$$
\begin{array}{l}
     |N_1T_1-N_iT_i|\le \frac{T_i}{6} \quad \mbox{with} \quad N_1=q \left\lfloor \frac{\hat T}{T_1}\right\rfloor+q, N_i=p_i \left\lfloor \frac{\hat T}{T_i}\right\rfloor+q 
\end{array}
$$
and $N_1T_1>\hat T$. Fixed $\vx_0$ and let $\omega_i:=N_1T_1-N_iT_i\in[-T_i/6,T_i/6], i\in[n]$, we have
$$
\begin{array}{l}
     \sum_{i=1}^{n} x_{0,i}^2\cos \left(2 \sqrt{\lambda_{i}} (N_1T_1)\right)=\sum_{i=1}^{n} x_{0,i}^2\cos \left(2 \sqrt{\lambda_{i}} (N_iT_i+\omega_i)\right)=\sum_{i=1}^{n} x_{0,i}^2\cos \left(2 \sqrt{\lambda_{i}} \omega_i\right)
\end{array}
$$
Applying Lemma~\ref{lemma:epin} with this $(\vx_0,\omega)$, there exists a time $\bar t\in(0,\bar t_{max}]$ such that
 \begin{equation}
\label{eqn:hat-G-case-I-new}
        \begin{array}{l}
             \sum_{i=1}^{n} x^2_{i} \cos \left(2 \sqrt{\lambda_{i}}( \bar t+\omega_i)\right)\le -\|\vx\|^2\epsilon_n, %\quad\mbox{and}\quad \bar t\leq \bar t_{max},
        \end{array}
 \end{equation}
where $\bar t_{max}$ obtained in \Cref{lemma:epin} is independent with $\vx_0$.
Take $t_{\vx_0} :=  N_1T_1 + \bar t$, then we have $\hat T<N_1T_1\le t_{\vx_0}\le  N_1T_1+\bar t_{max}$ and $t_{\vx_0}=N_iT_i+\omega_i+\bar t$, then \eqref{eqn:hat-G-case-I-new} and \eqref{approx} imply
$$
    \begin{array}{rl}
           S(t_{\vx_0})  & \le \sum_{i=1}^n x_{0,i}^2 \cos(2\sqrt{\lambda_i}(N_iT_i+\omega_i+\bar t) + \sum_{i=1}^n x_{0,i}^2 |G(\sqrt{\lambda_i}t_{\vx_0}) - \cos(2\sqrt{\lambda_i}t_{\vx_0})| \\
           & \leq -\|\vx_0\|^2\epsilon_n +  \|\vx_0\|^2\epsilon_n = 0.
    \end{array}
$$
The proof is completed with \eqref{Gt} and $ T:=N_1T_1+\bar t_{max}$,
where $N_1,T_1$ and $\bar t_{max}$ are independent with $\vx_0$.
\end{proof}

\section{Conclusions}
\label{sec:con}
Restart techniques are frequently employed in accelerated gradient methods and have demonstrated promising numerical performance across various problem domains. This paper focuses on the mathematical analysis of the gradient restarted APG method, a widely used adaptive restart approach. We first establish the global linear convergence of both the original and gradient restarted APG methods for solving strongly convex composite optimization problems.
Moreover, to highlight the benefits of the gradient restart scheme, we prove the global linear convergence for the trajectory of the ODE model for both quadratic and convex objective functions. This is a property that was shown to be absent in \cite{su2016differential}.
Building upon these results, this work provides a theoretical understanding of the inherent properties of the gradient restarting within the accelerated gradient methods, which are consistently observed in numerical experiments.
While our results partially address the open question presented in \cite{su2016differential}, it is of interest to fully resolve this issue for the general strongly convex cases.

%%%%%%%%%%%%%%%%%%%%%%%%%%%%%%%%%%%%%%%%%%%%%%%%%%%%%%%
%%% Appendix sections. žœÂŒÕÂœÚ, ·Ç±ØÑ¡
%%%%%%%%%%%%%%%%%%%%%%%%%%%%%%%%%%%%%%%%%%%%%%%%%%%%%%%
\begin{appendix}

%\appendix
\section{A refined analysis for the smooth scenario}
%\label{sec:nagc}

This appendix considers the scenario where $g$ is vacuous and the APG method reduces to the iteration \eqref{NAG}. 
Even the results in \Cref{sec: discrete} (\Cref{thm: APG}, \Cref{thm: apgseq}, \Cref{thm: reapgc}, and \Cref{coro: bouned reapg}) apply to this smooth scenario, further using the fact that $g\equiv 0$ can obtain a finer convergence rate analysis with or without gradient restart. 
For convenience, we make the following assumption.
\begin{assumption}
\label{ass:smooth}
Assume that $f\in\mathcal S_{\mu, L}$, $g\equiv 0$ with $L>\mu>0$, and $s\in(0,1/L)$. Moreover,  $\vx^\star$ is the unique minimizer of $f$ and $f^\star = f(\vx^\star)$.
\end{assumption}

Using the Lyapunov function in \eqref{lyap: APG}, we refine the upper bound \eqref{upper bound: lyapk+1 apg} in the smooth scenario as the following lemma.

\begin{lemma}
\label{lemma: upp nag}
Under Assumption \ref{ass:smooth}, let $\{\vx_k\}$ and $\{\vy_k\}$ be the sequences generated by iteration \eqref{NAG} with step size $s\in(0,1/L)$ and $\{\cE_k\}$ be the Lyapunov sequence defined by \eqref{lyap: APG}. 
Then, it holds that
\begin{equation}
\label{upper bound: lyapk}
\begin{array}{l}
\cE_k\leq  \frac{ s(t_{k+1}-1)t_{k+1}(1+\mu/a)}{2\mu} \left\|\nabla f(\vy_k)\right\|^2
+\frac{1+1/b}{2}\left\|\vy_{k}-{\vx^\star}\right\|^{2}
%\\
+\frac{(1+b)(t_{k+1}-1)^2+s(t_{k+1}-1)t_{k+1}(a+L)}{2}\left\| \vy_k-\vx_k\right\|^{2}
\end{array}
\end{equation}
for all $a>0,b>0$ and $k\ge0$.
\end{lemma}
\begin{proof}
    Note that the auxiliary function $G_s$ reduces to $\nabla f$ when $g$ is vacuous. Since $f\in\mathcal{S}_{\mu,L}$, we have $f(\vx)-f^\star\ge\mu/2\|\nabla f(\vx)\|^2$. Taking \eqref{ykxk} into the first part of \eqref{lyap: APG}, one has for any $a>0$ that
$$
\begin{array}{ll}
\cE_k^p
&\le
s(t_{k+1}^2-t_{k+1})\big[\nabla f(\vy_k)^{T}(\vx_k-\vy_k)+\frac{L}{2}\left\| \vx_k-\vy_k\right\|^2\big]
+\frac{ s(t_{k+1}-1)t_{k+1}}{2\mu} \left\|\nabla f(\vy_k)\right\|^2
\\[1mm]
&\le
s(t_{k+1}^2-t_{k+1})\big[\big(\frac{1}{2a}+\frac{1}{2\mu}\big)\left\|\nabla f(\vy_k)\right\|^2
+\frac{a}{2}\left\| \vy_k-\vx_k\right\|^{2}
+\frac{L}{2}\left\| \vx_k-\vy_k\right\|^2\big],
\end{array}
$$
where the last inequality comes from the Cauchy–Schwarz inequality. Similarly, the second part of \eqref{lyap: APG} satisfies
$$
\begin{array}{l}
     \cE_k^m
\le
\frac{1+b}{2}(t_{k+1}-1)^2\left\| \vy_k-\vx_k\right\|^{2}
+\frac{1+1/b}{2}\left\|\vy_{k}-{\vx^\star}\right\|^{2}
\end{array}
$$
for any $b>0$. Consequently, summing the above two inequalities together implies \eqref{upper bound: lyapk}, which completes the proof.
\end{proof}
Based on \Cref{lemma: upp nag}, we have a new global R-linear convergence of \Cref{algoapgc} with vacuous $g$.

\begin{theorem}
\label{thmnagcle}
Under \Cref{ass:smooth}, let $\{\vx_k\}$ and $\{\vy_k\}$ be the sequences generated by iteration \eqref{NAG} and $\cE_k$ be defined in \eqref{lyap: APG}. Then, there exists a positive sequence $\{\eta_k\}_{k=1}^{\infty}$ such that 
\begin{equation}
\label{linear: nagle}
\begin{array}{l}
     \cE_{k}\leq \eta_{k}\cE_{k-1},\quad \mbox{and} \quad f\left(\vx_{k}\right)-f^{\star}\leq
     \frac{\prod_{i=1}^{k}\eta_i}{2s(t_{k+1}-1)t_{k+1}}\cdot\left\|\vx_{0}-\vx^{\star}\right\|^{2}\quad \mbox{for all } k\ge1
\end{array} 
   \end{equation}
Specifically, $\{\eta_i\}_{i=0}^\infty$ is a strictly increasing sequence with $\eta _1 = 1-\mu s$ and $\bar\eta:=\lim\limits_{i\to+\infty}\eta_{i}\leq1- \frac{(1-L s)\mu s}{1+\max\{\frac{\mu}{L},\frac{1}{8}\}}$.
% \begin{equation}
%     \label{def: bareta}
%     \begin{array}{l}
%         \eta _1 = 1-\mu s, \mbox{ and }\ \bar\eta:=\lim\limits_{i\to+\infty}\eta_{i}\leq1- \frac{(1-L s)\mu s}{1+\max\{\frac{\mu}{L},\frac{1}{8}\}}<\rho
%     \end{array}
% \end{equation}
\end{theorem}

\begin{proof}
    Given $t\ge 1$ and $s\in(0,1/L)$, let $\Phi_{t}:\RR_{++}^2 \to \RR_{++}$ be the function 
    $$
    \begin{array}{l}
\Phi_t(a,b):=
\max\big\{ \underbrace{(t -1)(1+\mu/a)/[t(1-sL)]}_{\phi_{t,1}(a)},
\underbrace{(1+b)(t-1)/{t}+s(a+L)}_{\phi_{t,2}(a,b)},
\underbrace{(1+1/b)/{t}}_{\phi_{t,3}(b)}
\big\},
\end{array}
$$
%     $$
%     \begin{array}{l}
% \Phi_t(a,b):=
% \max\big\{ \underbrace{\frac{(t -1)(1+\mu/a)}{t(1-sL)}}_{\phi_{t,1}(a)},
% \underbrace{\frac{(1+b)(t-1)}{t}+s(a+L)}_{\phi_{t,2}(a,b)},
% \underbrace{\frac{1+1/b}{t}}_{\phi_{t,3}(b)}
% \big\},
% \end{array}
% $$
% with three positive functions defined by
% $$
% \begin{array}{l}
%      \phi_1(a):=\frac{(t -1)(1+\mu/a)}{t(1-sL)},
% \quad
% \phi_2(a,b):=\frac{(1+b)(t-1)}{t}+s(a+L),
% \quad
% \phi_3(b):=\frac{1+1/b}{t}.
% \end{array}
% $$
and define $ E(t):= \inf_{a>0,b>0}\Phi_{t}(a,b)$. From \Cref{lemma: upp nag} and 
\Cref{lemma: diff apg}, we know that $\cE_{k}\le \frac{E(t_{k+1})}{\mu s}(\cE_{k}-\cE_{k+1})$. Next, we work on the monotonicity of $E(t)$. Firstly, one has $E(1)=1$ when $t=1$ and $E(t)\le\Phi_t(1/s,1)<\frac{3}{1-Ls}$ when $t>1$. Given $t>1$, note that $\Phi_{t}(a,b)\ge\frac{4}{  1-Ls }$ if $(a,b)\notin [\frac{\mu}{4},4L]\times[\frac{1}{4},4]$, thus the minimum of $\Phi_{t}(a,b)$ should be attained in $\RR_{++}^2$. Moreover, according to the monotonicity properties of $\phi_{t,1}$, $\phi_{t,2}$, and $\phi_{t,3}$ with respect to $a$ and $b$, there exists a unique solution $(A(t),B(t))=\argmin_{a>0,b>0}\Phi_{t}(a,b)$, which also uniquely solve the equations $\phi_1(a)=\phi_2(a,b)=\phi_3(b)$. Note that 
$E(t)=\phi_3(B(t))=\phi_2(A(t),B(t))$ for all $t>1$, thus one has $tE(t)-1=1/B(t)$ and
$$
\begin{array}{l}
E(t)=\frac{(1+1/B(t))(t-1)}{t /B(t)}+(A(t)+L)s
=
\frac{(t-1)}{tE(t)-1} E (t)+(A(t)+L)s.
\end{array}
$$ 
This equality implies that $\frac{t-1}{tE(t)-1}<1$,
which yields $E(t)>1$ for all $t>1$. Define two functions as
\begin{equation*}
\left\{\begin{array}{cc}
\vartheta (t,a,e):= \frac{t-1}{t}(1+\frac{\mu}{a})-(1-sL)e,\\
\xi       (t,a,e):= \frac{t-1}{t}\frac{te}{te-1}+(a+L)s-e,
\end{array}
\right.
\quad
\mbox{for all } t>1, a>0, e>1.
\end{equation*}
Then $(t, A(t), E(t)$ is the solution of the equations $\vartheta (t,a,e)=
\xi(t,a,e)=0$. By direct calculations, we have
$$
\begin{array}{c}
     E'(t)=-\begin{vmatrix}
\frac{\partial (\vartheta,\xi) }{\partial (a,t)}
\end{vmatrix}\Big /\begin{vmatrix}
\frac{\partial (\vartheta,\xi) }{\partial (a,e)}
\end{vmatrix}
\end{array}>0 \quad\mbox{with}\quad
a=A(t)>0\ \mbox{and}\  e=E(t)>1.
$$
This, together with $E(1)=1$ demonstrates the strict increasing property of $E(t),t\ge1$. 
Moreover, a similar analysis shows that $a_{\infty}:=\argmin_{a>0}
\max\left\{
\frac{1+\mu/a}{1-sL},
1+s(a+L)
\right\}$ is well-defined. Therefore, we have
$$
\begin{array}{ll}
    E(t)&\le\Phi_{t}\left(a_\infty,\frac{1}{t-1}\right)
\le
\max\left\{
\frac{1+\mu/a_{\infty}}{(1-sL)},
1+s(a_\infty+L),
1
\right\}\\
&= \min_{a>0} \max\left\{\frac{(1+\mu/a)}{(1-sL)},1+s(a+L)\right\}\\
&\le\frac{1}{ 1-sL }\max\left\{1+\mu/L,(1+2L s)(1-sL)\right\}=\frac{1+\max\{\mu/L, 1/8\}}{ 1-L s }.
\end{array}
$$
The proof is completed with $\cE_{k}\le \eta_k \cE_{k-1}$ where $\eta_k:=1-\mu s/E(t_k)$.
\end{proof}

Then we further establish the linear convergence of the iteration sequence generated by the iteration \eqref{NAG}.

\begin{proposition}
\label{thm: nagseq}
    Under \Cref{ass:smooth}, let $\{\vx_k\}$ and $\{\vy_k\}$ be the sequences generated by the iteration \eqref{NAG} with $s<1/L$. Then, there exists a positive sequence $\{\eta_k\}_{k=0}^{\infty}$ such that one has
\begin{equation}\label{xkx0}
\begin{array}{c}
     \|\vx_k-\vx^\star\|^2
 \leq
 \max\left\{ \frac{1}{t_k}
 , 1-\mu s (t_k-1) \right\}
 \prod_{i=0}^{k-1}\eta_{i}
 \left\|\vx_{0}-\vx^\star\right\|^{2}
\end{array}
     \end{equation}
     for all $k\geq 1$. Specifically, $\{\eta_i\}_{i=0}^\infty$ is a strictly increasing sequence and satisfies
\begin{equation}\label{def: eta}
\begin{array}{l}
     \eta_{0}= 1-\frac{2\mu L s}{\mu +L},~\eta _1 = 1-\mu s, \mbox{ and }\ \bar\eta:=\lim\limits_{i\to+\infty}\eta_{i}\leq1- \frac{(1-L s)\mu s}{1+\max\{\frac{\mu}{L},\frac{1}{8}\}}<\rho. 
\end{array}
\end{equation} 

 \end{proposition}

 \begin{proof}
     We prove \eqref{xkx0} by induction. According to \cite[Theorem 2.1.15]{Nesterov2004IntroductoryLO} we know that
\begin{equation}
\label{GD descent}
\begin{array}{c}
     \left\|\vx_{k+1}-\vx^\star\right\|^2\leq (1-\frac{2\mu L s}{\mu +L})\left\|\vy_k-\vx^\star\right\|^2 \quad\mbox{for all}\quad k\ge 0.
\end{array}
\end{equation}
Then the inequality \eqref{xkx0} holds when $k=1$ with $\eta_{0}= 1-\frac{2\mu L s}{\mu +L}$. Assume that \eqref{xkx0} holds for a certain $k\ge2$, then $\vx_k$ satisfies $\left\|\vx_k-\vx^\star\right\|^2 \leq
\prod_{i=0}^{k-1}\eta_{i}
\left\|\vx_{0}-\vx^\star\right\|^{2}\leq
\prod_{i=1}^{k}\eta_{i}
\left\|\vx_{0}-\vx^\star\right\|^{2}$. With the similar analysis as the inequality \eqref{apg ykxkx0} in \Cref{thm: apgseq}, the inequality \eqref{linear: nagle} implies that there exists a positive sequence $\{\eta_i\}_{i=1}^\infty$ satisfying \eqref{def: eta} such that
\begin{equation*}%\label{ykxkx0}
\begin{array}{c}
     \left\|\vy_{k}-\vx^\star\right\|^2
\leq \left(1-\frac{1}{t_{k+1}}
-\mu s (t_{k+1}-1)\right)
\left\|\vx_{k}-\vx^\star\right\|^2
+\frac{1}{t_{k+1}}
\prod_{i=1}^{k}\eta_{i}
{\left\|\vx_{0}-\vx^\star\right\|^{2}},
\end{array}
\end{equation*}
for all $k\ge1$. %where $\{\eta_i\}_{i=1}^\infty$ is 
Taking \eqref{GD descent} into account, the proof is completed by
$$
\begin{array}{ll}
     &\left\|\vx_{k+1}-\vx^\star\right\|^2\leq \eta_0\left\|\vy_k-\vx^\star\right\|^2\\
\leq &\max \left\{0,
1-\frac{1}{t_{k+1}}-\mu s(t_{k+1}-1)\right\}
\prod_{i=0}^{k}\eta_{i}
\left\|\vx_{0}-\vx^\star\right\|^{2}
+\frac{1}{t_{k+1}}
\prod_{i=0}^{k}\eta_{i}
\left\|\vx_{0}-\vx^\star\right\|^{2}\\
\leq &\max \left\{\frac{1}{t_{k+1}},
1-\mu s(t_{k+1}-1)\right\}
\prod_{i=0}^{k}\eta_{i}
\left\|\vx_{0}-\vx^\star\right\|^{2}.
\end{array}
$$
 \end{proof}

Furthermore, we have a sharper convergence result for \Cref{al: reapgc} without the nonsmooth term $g$.

\begin{theorem}\label{thm: renagc}
Under \Cref{ass:smooth}, let $\{\vx_k\}$ and $\{\vy_k\}$ be the sequences generated by \Cref{al: reapgc}. Then there exists a positive sequence $\{\eta_{i}\}_{i=0}^{\infty}$ satisfying \eqref{def: eta} such that
\begin{equation}\label{linear-nag-restart}
\begin{array}{c}
     \|\vx_k-\vx^\star\|^2
    \leq
   \prod_{i=0}^{k-1}\eta_{i}\cdot \|\vx_0-\vx^\star\|^2 \quad \mbox{for all}\quad k\ge 1.
\end{array}
\end{equation}
\end{theorem}

\begin{proof}
Let $j=k-S_{m_k}$ with $\{S_i\}$ defined in \eqref{def: Ki} and $m_k$ defined in \eqref{def: m_k}. According to the increasing property of $\eta_k$ and the inequality \eqref{xkx0}, one has for all $k\ge1$ that
$$
\begin{array}{l}
        \left\|\vx_{k}-\vx^\star\right\|^{2}
     \le
     \prod_{i=0}^{j-1}\eta_{i}
\|\vx_{S_{m_k}}-\vx^\star\|^2
\le
\prod_{i=0}^{j-1}\eta_{i}\cdot
\big(\prod_{l=1}^{m_k}
 \prod_{i=0}^{K_{l}-1}\big)\eta_{i}
\|\vx_0-\vx^\star\|^2  
\le
 \prod_{i=0}^{k-1}\eta_{i}
    \cdot \|\vx_0-\vx^\star\|^2.
\end{array}
$$
\end{proof}

\begin{remark}
    Note that $\eta_{0}=1-\frac{2\mu L s}{\mu +L}, \eta_{1}=1-\mu s$ and $\eta_{i}$ strictly increases to $ \bar\eta<\rho$, one has $\prod_{i=0}^{k-1}\eta_{i}<\prod_{i=1}^{k}\eta_{i}<\left( 1-\mu s\right)\bar\eta^{k-1}<\rho^k
    $. Consequently, then all the convergence rates obtained in \Cref{thmnagcle}, \Cref{thm: nagseq}, and \Cref{thm: renagc} are sharper than those in the nonsmooth scenario discussed in \Cref{sec: discrete}.
\end{remark}

Finally, an improvement for restarting is established as follows, given that the restart intervals are uniformly bounded.

% When further assuming the restart intervals are uniformly bounded, APG with gradient restarting and $g=0$ achieves a sharper convergence result as shown in the following corollary.
\begin{corollary}
\label{coro: bouned}
Under \Cref{ass:smooth}, let $\{\vx_k\}$ and $\{\vy_k\}$ be the sequences generated by \Cref{al: reapgc}. Assume the restart intervals $\{K_i\}$ are uniformly bounded from above by some integer $\overline{\cal K}>2$. Then there exists a positive sequence $\{\hat\eta_{i}\}_{i=0}^{\infty}$ such that
$$
\begin{array}{c}
     \|\vx_k-\vx^\star\|^2
    \leq
   \prod_{i=0}^{k-1}\hat\eta_{i}\cdot \|\vx_0-\vx^\star\|^2 \quad \mbox{for all }k\ge 1,
  \end{array}
$$
where $\hat \eta_i\leq \eta_i,i\ge0$ and $\hat \eta:=\sup_i \hat \eta_i<\min\{\hat \rho, \bar\eta\}$.
\end{corollary}

\begin{proof} 
Since we assume the restart interval $K_l\leq \overline{\cal K}$, then we have $\prod_{i=0}^{K_l-1}\eta_{i}<\eta_{\overline{\cal K}}^{K_l},\forall~l\ge 1$ and $\prod_{i=0}^{j-1}\eta_{i}<\eta_{\overline{\cal K}}^{j}$.
Given $k\ge1$, with the same definition of $\vartheta_{\overline{\cal K}}$ and $\{\bar \theta_l\}$ in \eqref{def: vartheta}, we have
$$
\begin{array}{ll}
\|\vx_k-\vx^\star\|^2
&\leq
\left(\bar \theta_{j}\prod_{i=1}^{m_k}\bar \theta_{K_i}
\right)\cdot
\big(
\prod_{i=0}^{j-1}\eta_{i}\cdot
\prod_{l=1}^{m_k}
 \prod_{i=0}^{K_{l}-1}\eta_{i}
\big)
\|\vx_0-\vx^\star\|^2
\\[1mm]
&\leq 
\big(\vartheta_{\overline{\cal K}}\big)^{k-1}
\big[\eta_{\overline{\cal K}}^{j+\sum_{i=2}^{m_k} K_{i}}\prod_{i=0}^{K_{1}-1}\eta_{i}\big]
\|\vx_0-\vx^\star\|^2
\\[1mm]
&= 
\eta_{0}\left(\vartheta_{\overline{\cal K}}\right)^{k-1}
\big(\prod_{i=1}^{k-1}
\tilde{\eta}_{i,\overline{\cal K}}
\big)
%\left(\eta_{\overline{\cal K}}\right)^{k-m_k-1}
\|\vx_0-\vx^\star\|^2,
\leq 
\eta_{0} \prod_{i=1}^{k-1}\hat \eta_{i}
\|\vx_0-\vx^\star\|^2,
\end{array}
$$
where $\tilde{\eta}_{i,\overline{\cal K}}:=\min\{\eta_i,\eta_{\overline{\cal K}}\}$ and $\hat \eta_i:=\vartheta_{\overline{\cal K}}\cdot\tilde{\eta}_{i,\overline{\cal K}}$ for $i\ge1$. Note that $1-\frac{2\mu L s}{\mu +L}<\tilde{\eta}_{i,\overline{\cal K}}\le\eta_i<\rho$ and $\sup_i\tilde{\eta}_{i,\overline{\cal K}}=\eta_{\overline{\cal K}}<\lim_{i\to\infty}\eta_i=\bar\eta$, then the proof is completed with $\hat \eta_0=\eta_0$.
\end{proof}

\begin{remark}
    Although this section primarily examines the linear convergence of one sequence, all the iterative sequences $\{\vx_k\},\{\vy_k\}$ and the function value sequence $\{f(\vx_k)\}$ converge at a similar rate.
\end{remark}

\end{appendix}


\begin{thebibliography}{99}

\bibitem{beck2009fast}
Beck~A, Teboulle~M.
A fast iterative shrinkage-thresholding algorithm for linear inverse
  problems.
SIAM J Imaging Sci, 2009, 2:  183-202

\bibitem{Nesterov1983AMF}
Nesterov~Y.
A method for solving the convex programming problem with convergence rate
  $O(1/k^2)$.
Soviet Mathematics Doklady, 1983:  372-376

\bibitem{tseng2008accelerated}
Tseng~P.
On accelerated proximal gradient methods for convex-concave optimization.
\url{https://www.mit.edu/~dimitrib/PTseng/papers/apgm.pdf}, 2008

\bibitem{Lan2011PrimaldualFM}
Lan~G~H, Lu~Z~S, Monteiro~R D~C.
Primal-dual first-order methods with {${\cal O}(1/\epsilon)$}
  iteration-complexity for cone programming.
Math Program, 2011, 126: 1-29

\bibitem{chambolle2015convergence}
Chambolle~A, Dossal~C.
On the convergence of the iterates of the “fast iterative
  shrinkage/thresholding algorithm”.
J Optim Theory Appl, 2015, 166:  968-982

\bibitem{Nesterov2004IntroductoryLO}
Nesterov~Y.
Introductory Lectures on Convex Optimization-A Basic Course.
New York: Springer, 2004

\bibitem{ODonoghue2015AdaptiveRF}
O'Donoghue~B, Cand{\`e}s~E~J.
Adaptive Restart for Accelerated Gradient Schemes.
Found Comput Math, 2015, 15:  715-732

\bibitem{bao2018coherence}
Bao~C~L, Barbastathis~G, Ji~H, et~al.
Coherence retrieval using trace regularization.
SIAM J Imaging Sci, 2018, 11:  679-706

\bibitem{shang2017superfast}
Shang~J~W, Zhang~Z~Y, Ng~H~K.
Superfast maximum-likelihood reconstruction for quantum tomography.
Phys Rev A, 2017, 95:  062336

\bibitem{ito2017unified}
Ito~N, Takeda~A, Toh~K-C.
A unified formulation and fast accelerated proximal gradient method for
  classification.
J Mach Learn Res, 2017, 18:  510-558

\bibitem{su2016differential}
Su~W~J, Boyd~S~P, Cand{\`e}s~E~J.
A Differential Equation for Modeling {Nesterov's} Accelerated Gradient Method:
  Theory and Insights.
J Mach Learn Res, 2016, 17: 153:1-153:43

\bibitem{polyak1964some}
Polyak~B~T.
Some methods of speeding up the convergence of iteration methods.
Ussr computational mathematics and mathematical physics, 1964, 4:  1-17

\bibitem{kim2016optimized}
Kim~D, Fessler~J~A.
Optimized first-order methods for smooth convex minimization.
Math Program, 2016, 159: 81-107

\bibitem{drori2014performance}
Drori~Y, Teboulle~M.
Performance of first-order methods for smooth convex minimization: a novel
  approach.
Math Program, 2014, 145: 451-482

\bibitem{drori2017exact}
Drori~Y.
The exact information-based complexity of smooth convex minimization.
J Complex, 2017, 39:  1-16

\bibitem{Kim2017OntheConvergence}
Kim~D, Fessler~J~A.
On the Convergence Analysis of the Optimized Gradient Method.
J Optim Theory Appl, 2017, 172:  187–205.

\bibitem{jang2023computer}
Jang~U, Gupta~S~D, Ryu~E~K.
Computer-assisted design of accelerated composite optimization methods:
  {OptISTA}.
arXiv:2305.15704, 2023

\bibitem{Nesterov2013GradientMF}
Nesterov~Y.
Gradient methods for minimizing composite functions.
Math Program, 2013, 140:  125-161

\bibitem{siegel2021accelerated}
Siegel~J~W.
Accelerated First-Order Methods: Differential Equations and Lyapunov Functions. arXiv:1903.05671, 2019

\bibitem{aujol2022convergence}
Aujol~J-F, Dossal~C, Rondepierre~A.
Convergence rates of the heavy ball method for quasi-strongly convex
  optimization.
SIAM J Optim, 2022, 32:  1817-1842

\bibitem{Aujol2022ConvergenceRO}
Aujol~J-F, Dossal~C, Rondepierre~A.
Convergence rates of the Heavy-Ball method under the Łojasiewicz
  property.
Math Program, 2022, 198:  195-254
%\url{https://api.semanticscholar.org/CorpusID:246171450}

\bibitem{park2021factorsqrt2}
Park~C, Park~J, Ryu~E~K.
Factor-2 acceleration of accelerated gradient methods.
Applied Mathematics \& Optimization, 2023, 88: ~77

\bibitem{Taylor2021AnOG}
Taylor~A~B, Drori~Y.
An optimal gradient method for smooth strongly convex minimization.
Math Program, 2021, 199:  557-594
%\url{https://api.semanticscholar.org/CorpusID:233477848}

\bibitem{ochs2014ipiano}
Ochs~P, Chen~Y~J, Brox~T, et~al.
iPiano: Inertial proximal algorithm for nonconvex optimization.
SIAM J Imaging Sci, 2014, 7~(2):  1388-1419

\bibitem{Chambolle2016AnIT}
Chambolle~A, Pock~T.
An introduction to continuous optimization for imaging.
Acta Numer, 2016, 25: 161 - 319
%\url{https://api.semanticscholar.org/CorpusID:124208295}

\bibitem{Fercoq2019AdaptiveRO}
Fercoq~O, Qu~Z.
Adaptive restart of accelerated gradient methods under local quadratic growth
  condition
IMA J Numer Anal, 2019, 39:  2069-2095

\bibitem{Lin2014AnAA}
Lin~Q~H, Xiao~L.
An adaptive accelerated proximal gradient method and its homotopy continuation
  for sparse optimization.
Comput Optim Appl, 2014, 60:  633-674

\bibitem{tao2016local}
Tao~S, Boley~D, Zhang~S~Z.
Local linear convergence of {ISTA} and {FISTA} on the {LASSO} problem.
SIAM J Optim, 2016, 26:  313-336

\bibitem{liang2017activity}
Liang~J~W, Fadili~J, Peyr{\'e}~G.
Activity Identification and Local Linear Convergence of Forward-Backward-type
  Methods.
SIAM J Optim, 2017, 27:  408-437

\bibitem{li2024linear}
Li~B~W, Shi~B, Yuan~Y~X.
Linear Convergence of Forward-Backward Accelerated Algorithms without Knowledge
  of the Modulus of Strong Convexity.
SIAM J Optim, 2024, 34:  2150-2168

\bibitem{liang2022improving}
Liang~J~W, Luo~T, Schonlieb~C-B.
Improving “fast iterative shrinkage-thresholding algorithm”: Faster,
  smarter, and greedier.
SIAM J Sci Comput, 2022, 44:  A1069-A1091

\bibitem{Aujol2023FISTAIA}
Aujol~J-F, Dossal~C, Rondepierre~A.
FISTA is an automatic geometrically optimized algorithm for strongly convex
  functions.
Math Program, 2023: 1-43
%\url{https://api.semanticscholar.org/CorpusID:250478830}

\bibitem{beck2009fastGradient}
Beck~A, Teboulle~M.
Fast gradient-based algorithms for constrained total variation image denoising
  and deblurring problems.
IEEE Trans Image Process, 2009, 18:
  2419-2434

\bibitem{Wang2022ConvergenceRO}
Wang~T, Liu~H~W.
Convergence Results of a New Monotone Inertial Forward–Backward Splitting
  Algorithm Under the Local {H{\"o}lder} Error Bound Condition.
Appl Math Optim, 2022, 85: 7

\bibitem{powell1977restart}
Powell~M J~D.
Restart procedures for the conjugate gradient method.
Math Program, 1977, 12: 241-254

\bibitem{Saad1986GMRESAG}
Saad~Y, Schultz~M~H.
GMRES: a generalized minimal residual algorithm for solving nonsymmetric linear
  systems.
SIAM J Sci Stat Comput, 1986, 7:  856-869
%\url{https://api.semanticscholar.org/CorpusID:18390597}.

\bibitem{zhang2022efficient}
Zhang~G~J, Yuan~Y~C, Sun~D~F.
An Efficient {HPR} Algorithm for the Wasserstein Barycenter Problem with
  $O({Dim(P)}/\varepsilon)$ Computational Complexity. arXiv:2211.14881, 2022

\bibitem{Zhang2018GloballyCT}
Zhang~J~Z, O'Donoghue~B, Boyd~S.
Globally convergent type-{I} Anderson acceleration for nonsmooth fixed-point
  iterations.
SIAM J Optim, 2020, 30:  3170-3197

\bibitem{Nemirovski2007EFFICIENTMI}
Nemirovski~A.
Efficient methods in convex programming.
\url{https://api.semanticscholar.org/CorpusID:116771466}, 2007

\bibitem{alamo2019restart}
Alamo~T, Limon~D, Krupa~P.
Restart FISTA with global linear convergence.
In 2019 18th European Control Conference (ECC),  IEEE, 2019

\bibitem{alamo2019gradient}
Alamo~T, Krupa~P, Limon~D.
Gradient based restart FISTA.
In 2019 IEEE 58th Conference on Decision and Control (CDC), IEEE, 2019

\bibitem{alamo2022restart}
Alamo~T, Krupa~P, Limon~D.
Restart of accelerated first-order methods with linear convergence under a
  quadratic functional growth condition.
IEEE Trans Autom Control, 2022, 68: 612-619

\bibitem{aujol2022fista}
Aujol~J-F, Dossal~C~H, Labarri{\`e}re~H, et~al.
FISTA restart using an automatic estimation of the growth parameter.
\url{https://hal. archives-ouvertes.fr/hal-03153525v3/file/Adaptive restart.pdf}, 2022

\bibitem{aujol2024parameter}
Aujol~J-F, Calatroni~L, Dossal~C, et~al.
Parameter-free FISTA by adaptive restart and backtracking.
SIAM J Optim, 2024, 34: 3259-3285

\bibitem{roulet2020sharpness}
Roulet~V, d'Aspremont~A.
Sharpness, Restart and Acceleration.
SIAM J Optim, 2020, 30: 262-289

\bibitem{Renegar2022ASN}
Renegar~J, Grimmer~B.
A Simple Nearly Optimal Restart Scheme For Speeding Up First-Order Methods.
Found Comput Math, 2022, 22: 211-256

\bibitem{becker2011templates}
Becker~S~R, Cand{\`e}s~E~J, Grant~M~C.
Templates for convex cone problems with applications to sparse signal
  recovery.
Math Program Comput, 2011, 3: 165-218

\bibitem{moursi2023accelerated}
Moursi~W~M, Pavlovic~V, Vavasis~S~A.
Accelerated gradient descent: A guaranteed bound for a heuristic restart strategy. arXiv:2310.07674, 2023

\bibitem{watson1995treatise}
Watson~G~N.
A treatise on the theory of {Bessel} functions.
Cambridge: Cambridge Univ Press, 1995

\bibitem{schmidt1980diophantine}
Schmidt~W~M.
Diophantine approximation.
Springer, 1980



\end{thebibliography}
\end{document}